\documentclass[a4paper, 11pt]{amsart}
\usepackage{amsmath, amssymb, amsfonts, amsthm, enumerate, mathtools, tikz-cd, hyperref,stmaryrd,flushend}
\usepackage[margin=1.2in]{geometry}

\newtheorem{thm}{Theorem}[section]
\newtheorem{lem}[thm]{Lemma}

\newtheorem{rem}[thm]{Remark}

\newtheorem{prop}[thm]{Proposition}
\newtheorem{cor}[thm]{Corollary}

\newtheorem{setup}[thm]{Set-up}

\def\QQ{\mathbb{Q}}
\def\ZZ{\mathbb{Z}}
\def\FF{\mathbb{F}}
\def\TT{\mathbb{T}}
\def\cO{\mathcal{O}}
\def\cF{\mathcal{G}}
\def\frob{\text{Frob}}
\def\chibar{\bar\chi}

\def\ord{\text{ord}}
\def\ad{\text{Ad}}

\newcommand{\GL}{\mathrm{GL}}

\def\Ga1{\Gamma_1}
\def\rhob{\bar\rho}

\def\tr{\operatorname{tr}}

\def\univ{\text{univ}}

\def\ps{\text{pd}}

\DeclareFontFamily{U}{wncy}{}
\DeclareFontShape{U}{wncy}{m}{n}{<->wncyr10}{}
 \DeclareSymbolFont{mcy}{U}{wncy}{m}{n}
 \DeclareMathSymbol{\Sh}{\mathord}{mcy}{"58} 

\begin{document}
\baselineskip 18pt

\title{Non-optimal levels of some reducible mod $p$ modular representations}
\author{Shaunak V. Deo} 
\email{shaunakdeo@iisc.ac.in}
\address{Department of Mathematics, Indian Institute of Science, Bangalore 560012, India}
\date{}


\begin{abstract}
Let $p \geq 5$ be a prime, $N$ be an integer not divisible by $p$, $\bar\rho_0$ be a reducible, odd and semi-simple representation of $G_{\mathbb{Q},Np}$ of dimension $2$ and $\{\ell_1,\cdots,\ell_r\}$ be a set of primes not dividing $Np$.
After assuming that a certain Selmer group has dimension at most $1$, we find sufficient conditions for the existence of a cuspidal eigenform $f$ of level $N\prod_{i=1}^{r}\ell_i$ and appropriate weight lifting $\rhob_0$ such that $f$ is new at every $\ell_i$.
Moreover, suppose $p \mid \ell_{i_0}+1$ for some $1 \leq i_0 \leq r$.
Then, after assuming that a certain Selmer group vanishes, we find sufficient conditions for the existence of a cuspidal eigenform of level $N\ell_{i_0}^2 \prod_{j \neq i_0} \ell_j$ and appropriate weight which is new at every $\ell_i$ and which lifts $\rhob_0$.
As a consequence, we prove a conjecture of Billerey--Menares in many cases.
\end{abstract}

\maketitle

\section{Introduction}
Let $p \geq 5$ be a prime and $N$ be an integer not divisible by $p$.
Let $f$ be a newform of level $N$ and weight $k \geq 2$ with Fourier coefficients in $\overline{\QQ_p}$ and denote by $\bar\rho_f$ the residual mod $p$ Galois representation attached to $f$.
To be precise, let $\cO_f$ be the ring of integers in the finite extension of $\QQ_p$ generated by the Hecke eigenvalues of $f$ over $\QQ_p$ and let $\varpi_f$ be its uniformizer.
From the works of Eichler-Shimura and Deligne, there is a $p$-adic Galois representation $$\rho_f :  \text{Gal}(\overline{\QQ}/\QQ) \to \GL_2(\cO_f)$$ associated to $f$ which is absolutely irreducible and unramified outside primes dividing $Np$.
Then $\rhob_f$ is the \emph{semi-simplification} of the mod $p$ Galois representation $\rho_f \pmod{\varpi_f}$.

In this setting, one can ask some natural questions: does there exist a newform $g$ of level $M \neq N$ and weight $k$ such that $\rhob_g \simeq \rhob_f$? 
If yes, what are all the levels at which such newforms occur?

In his seminal work, Ribet \cite{Ri} studied these questions and established level raising for cuspidal eigenforms of level $\Gamma_0(N)$ and weight $2$.
In particular, he proved that if $f$ is a cuspidal eigenform of level $\Gamma_0(N)$ and weight $2$, $\ell$ is a prime not dividing $Np$ and $\rhob_f$ is absolutely irreducible, then there exists a cuspidal eigenform $g$ of level $\Gamma_0(N\ell)$ and weight $2$ such that $g$ is new at $\ell$ and $\rhob_g \simeq \rhob_f$ if and only if $$\tr(\rhob_f(\frob_{\ell}))^2= (\ell+1)^2 \pmod{\varpi_f}.$$
Diamond \cite{Dia} generalized this result to establish level raising for eigenforms of weight $k \geq 2$.
The study of such level raising results (in the more general setting of automorphic forms) and their consequences has now become a central theme in number theory.

On the other hand, given an odd, continuous, semi-simple representation $$\rhob_0 :  \text{Gal}(\overline{\QQ}/\QQ) \to \GL_2(\FF)$$ over a finite field $\FF$ of characteristic $p$ and an integer $p \nmid N$ divisible by the Artin conductor of $\rhob_0$, Carayol \cite{Ca} determined necessary conditions for the existence of a newform of level $N$ lifting $\rhob_0$.
In other words, he classified all possible \emph{non-optimal} levels of newforms lifting $\rhob_0$.
Here we say that a newform $f$ lifts $\rhob_0$ if $\rhob_f \simeq \rhob_0$, where $\rhob_f$ is the semi-simple mod $p$ Galois representation attached to $f$ as above.
This leads us to a natural question: for which non-optimal levels does there exist a newform lifting $\rhob_0$?

This question was studied by Diamond and Taylor in \cite{DT}.
If $\rhob_0$ is absolutely irreducible, then they proved, under some mild assumptions, that for an appropriate weight and every non-optimal level, there exists a newform lifting $\rhob_0$ (see \cite[Theorem A]{DT}).
In fact, they generalized level raising results of Ribet and Diamond to multiple primes to get a newform for every non-optimal level lifting $\rhob_0$ (see \cite[Theorem B, C]{DT}).

The main aim of this article is to study this question for reducible $\rhob_0$'s and establish level raising results (in the spirit of Diamond--Taylor) for modular forms with residually reducible representations.
Note that the geometric techniques used by Diamond and Taylor to answer this question for absolutely irreducible $\rhob_0$'s do not work in the setting of reducible $\rhob_0$'s.
We use techniques from deformation theory of Galois representations, along with the modularity lifting theorems of Skinner--Wiles (\cite{SW2}) and the finiteness result of Pan (\cite[Theorem 5.1.2]{P}), to study this question.

\subsection{History} 
Very few results are known about level raising and non-optimal levels of modular forms with \emph{reducible} residual representations.
Before stating them, let us denote the mod $p$ cyclotomic character by $\omega_p$.

In his landmark work on Eisenstein ideal (\cite{M2}), Mazur proved that if $\ell \neq p$ is a prime, then there exists a newform of level $\Gamma_0(\ell)$ and weight $2$ lifting $1 \oplus \omega_p$ if and only if $p \mid \ell-1$.
In \cite{Y}, Yoo partially extended Mazur's result to squarefree levels.
To be precise, he gave sufficient conditions for the existence of a newform of level $\Gamma_0(N)$ with squarefree $N$ and weight $2$ lifting $1 \oplus \omega_p$ (see \cite[Theorem 1.3]{Y}).
He also formulated some necessary conditions for the existence of such newforms and proved that they are sufficient in some cases (see \cite[Section 6]{Y}).
Note that some of the results presented in \cite{Y} are proved by Ribet \cite{Ri2} (see \cite[Section 2]{Y}).
Similar results were also obtained by Wake and Wang-Erickson in \cite{WWE1} using different methods.
Mazur's results were extended to the setting of Hilbert modular forms by Martin \cite{Ma} and he also proved, in the modular forms case, some results similar to those of Yoo (see \cite[Theorem A]{Ma}).

On the other hand, in \cite{BM}, Billerey and Menares determined, for an even integer $4 \leq k < p-1$, all primes $\ell$ for which there exists a newform $f$ of level $\Gamma_0(\ell)$ and weight $k$ lifting $1 \oplus \omega_p^{k-1}$ (see \cite[Theorem 1]{BM}).
They also proposed a conjecture about all squarefree levels at which there exists a newform lifting $1 \oplus \omega_p^{k-1}$ (see \cite[Conjecture 3.2]{BM}).
Moreover, in a subsequent work (\cite{BM2}), they extended Mazur's result to other reducible representations.
In particular, if $\bar\rho_0$ is an odd, reducible, semi-simple mod $p$ Galois representation with Artin conductor $N_0$ such that there does not exist a newform of level $N_0$ lifting $\rhob_0$ and $\ell \nmid N_0p$ is a prime, then they gave necessary and sufficient conditions for the existence of a newform of level $N_0\ell$ and appropriate weight lifting $\rhob_0$ (see \cite[Theorem 2]{BM2}).
The main result of \cite{BM} was also proved in \cite{GP} and \cite{KKMS} using different methods.

In a recent work (\cite{LW}), Lang and Wake proved that if $\ell$ is a prime such that $p \mid \ell+1$, then there exists a newform of level $\Gamma_0(\ell^2)$ and weight $2$ lifting $1 \oplus \omega_p$ (see \cite[Theorem B]{LW}).
In our previous work (\cite{D2}), we obtained some results about non-optimal levels of newforms of weight $k > 2$ lifting a reducible $\rhob_0$ which are not covered by the results mentioned above (\cite[Theorem B]{D2}).
To the best of our knowledge, no other results are known about non-optimal levels of newforms lifting a reducible $\rhob_0$.

Note that the methods of Mazur (\cite{M2}), Ribet (\cite{Ri2}) and Yoo (\cite{Y}) can be termed as `geometric' as they study geometry of Jacobians of modular curves to prove their results.
On the other hand, Billerey--Menares (\cite{BM}, \cite{BM2}) and Wake--Lang compute the constant term of the relevant Eisenstein series at all the cusps of an appropriate modular curve to produce a cusp form having the desired properties.
So we can say that their methods are `analytic'.
The methods of \cite{GP} and \cite{KKMS}, which are different from those of \cite{BM}, are still analytic.
Finally, in \cite{WWE1} and \cite{D2}, deformation theory of Galois (pseudo-)representations is used to prove the level raising results.
So their methods can be called `algebraic'.

In the present article, we follow the approach of \cite{D2} to tackle the problem using algebraic methods.
Note that, in the language of $R=\TT$ theorems, a level raising result translates into the statement that a certain Hecke algebra $\TT$ is `big enough'.
Now the geometric and analytic methods study the properties of this Hecke algebra $\TT$ directly.
On the other hand, algebraic methods study an appropriate deformation ring $R$ and then relate it with the Hecke algebra.
So typically, algebraic methods prove that the deformation ring is big enough but that is not sufficient to conclude that the Hecke algebra is also big enough.
In this article, we combine algebraic methods with the modularity lifting theorems of Skinner--Wiles (\cite{SW2}) to overcome this obstruction.

\subsection{Set-up and Main results}
Before stating our main results, we will describe the setup with which we will be working throughout the article.

\begin{setup}
\label{setup}
 Let $N \geq 1$ be an integer such that $p \nmid N\phi(N)$. 
Let $G_{\QQ,Np}$ be the Galois group of the maximal extension of $\QQ$ unramified outside primes dividing $Np$ and $\infty$ over $\QQ$.
Let $\FF$ be a finite field of characteristic $p$.
Let $\rhob_0 : G_{\QQ,Np} \to \GL_2(\FF)$ be a continuous, odd representation such that $\rhob_0 = \chibar_1 \oplus \chibar_2$ for some continuous characters $\chibar_1, \chibar_2 : G_{\QQ,Np} \to \FF^{\times}$. Let $\chibar=\chibar_1 \chibar_2^{-1}$.
Denote by $\omega_p$ the mod $p$ cyclotomic character of $G_{\QQ,Np}$.
Let $N_0$ be the Artin conductor of $\rhob_0$ and let $\det(\rhob_0)=\bar\psi\omega_p^{k_0-1}$ with $\bar\psi$ is unramified at $p$.
Suppose the following conditions hold:
\begin{enumerate}
\item $\chibar_2$ is unramified at $p$,
\item\label{hyp2} $N_0 \mid N$,
\item\label{hyp4} $1 < k_0 < p$ and $\bar\chi|_{G_{\QQ_p}} \neq \omega_p^{-1}|_{G_{\QQ_p}}$,
\item\label{hyp3} If $\ell \mid N$, $p \mid \ell +1$ and $\chibar|_{G_{\QQ_{\ell}}} = \omega_p|_{G_{\QQ_{\ell}}}$, then $\ell^2 \mid N$.
\end{enumerate}
\end{setup}


For a prime $\ell$ and a representation $\rho$ of $G_{\QQ,Np}$, denote by $\rho|_{G_{\QQ_{\ell}}}$ the restriction of $\rho$ to the local Galois group at $\ell$ and let $H^1_{\{p\}}(G_{\QQ,Np},\rho) := \ker (H^1(G_{\QQ,Np},\rho) \to H^1(G_{\QQ_p},\rho))$.
If $f$ is a modular eigenform with Fourier coefficients in a finite extension of $\QQ_p$, then we denote by $\rho_f$ the $p$-adic Galois representation attached to it by Eichler--Shimura and Deligne.
We say that $\rho_f$ lifts $\rhob_0$ if the semi-simplification $\rhob_f$ of the corresponding residual mod $p$ Galois representation is isomorphic to $\rhob_0$.

We are now ready to state our main results. 
\begin{thm}
\label{thmb}
Suppose we are in the Set-up~\ref{setup} as above and $\dim(H^1_{\{p\}}(G_{\QQ,Np},\chibar^{-1})) = 1$.
Let $k>2$ be an integer such that $k \equiv k_0 \pmod{p-1}$ and $\ell_1,\cdots,\ell_r$ be primes such that $\ell_i \nmid Np$, $\chibar|_{G_{\QQ_{\ell_i}}} = \omega_p|_{G_{\QQ_{\ell_i}}}$ and $p \nmid \ell_i-1$ for all $1 \leq i \leq r$.
Then there exists an eigenform $f$ of level $N\prod_{i=1}^{r}\ell_i$ and weight $k$ such that $\rho_f$ lifts $\rhob_0$ and $f$ is new at $\ell_i$ for every $1 \leq i \leq r$.
\end{thm}

\begin{thm}
\label{thmc}
Suppose we are in the Set-up~\ref{setup} as above and $\ell_0$ is a prime such that $\ell_0 \nmid Np$, $\chibar|_{G_{\QQ_{\ell_0}}} = \omega_p^{-1}|_{G_{\QQ_{\ell_0}}}$ and $p \nmid \ell_0-1$.
Suppose the following hypotheses hold:
\begin{enumerate}
\item $H^1_{\{p\}}(G_{\QQ,Np},\chibar^{-1}) = 0$,
\item If $p \mid \ell_0+1$, then $\chibar \neq \omega_p$. 
\end{enumerate}
Let $k>2$ be an integer such that $k \equiv k_0 \pmod{p-1}$ and $\ell_1,\cdots,\ell_r$ be primes such that $\ell_i \nmid Np$, $\chibar|_{G_{\QQ_{\ell_i}}} = \omega_p|_{G_{\QQ_{\ell_i}}}$ and $p \nmid \ell_i-1$ for all $1 \leq i \leq r$.
Then there exists an eigenform $f$ of level $N\prod_{i=0}^{r}\ell_i$ and weight $k$ such that $\rho_f$ lifts $\rhob_0$ and $f$ is new at $\ell_i$ for every $0 \leq i \leq r$.
\end{thm}

\begin{rem}
Note that the Artin conductor $N_0$ of $\rhob_0$ satisfies Hypothesis~\eqref{hyp3} of Set-up~\ref{setup}.
Hence, if we take $N=N_0$ in Theorems~\ref{thmb} and \ref{thmc}, then the eigenforms that we obtain in Theorems~\ref{thmb} and \ref{thmc} are newforms.
\end{rem}

\begin{rem}
The proof of Theorem~\ref{thmc} implies that, under its hypotheses, there exists an eigenform $f$ of level $N\ell_0$ and weight $k$ such that $\rho_f$ lifts $\rhob_0$ and $f$ is new at $\ell_0$.
\end{rem}

In \cite{D2}, we have proved a result similar to Theorems~\ref{thmb} and \ref{thmc} (see \cite[Theorem B]{D2}).
However, there are a couple of key differences between \cite[Theorem B]{D2} and the theorems above which we will now explain.
Before moving ahead, we will establish some notation.
We say that a prime $\ell$ satisfies the level raising condition for $\chibar$ (resp.  for $\chibar^{-1}$) if $\ell \nmid Np$ and $\chibar|_{G_{\QQ_\ell}} = \omega_p|_{G_{\QQ_{\ell}}}$ (resp. $\chibar^{-1}|_{G_{\QQ_\ell}} = \omega_p|_{G_{\QQ_{\ell}}}$).

Firstly, in \cite[Theorem B]{D2}, we assume that a \emph{global Galois cohomology group of $\chibar$} is cyclic (i.e. $\dim(H^1(G_{\QQ,Np},\chibar)) =1$). On the other hand, in Theorems~\ref{thmb} and \ref{thmc}, we assume that a \emph{trivial at $p$ Selmer group of $\chibar^{-1}$} is either non-zero and cyclic ($\dim(H^1_{\{p\}}(G_{\QQ,Np},\chibar^{-1})) = 1$) or trivial ($H^1_{\{p\}}(G_{\QQ,Np},\chibar^{-1}) = 0$).
No direct relation between these Galois cohomology groups has been established so far.
However, the cyclicity hypothesis appearing in Theorem~\ref{thmb} is milder than the cyclicity hypothesis appearing in \cite[Theorem B]{D2} (see \S\ref{assumesec} for more details).

Secondly, in \cite[Theorem B]{D2}, we obtain simultaneous level raising at primes $\ell_i$'s which satisfy the level raising condition for $\chibar^{-1}$.
On the other hand, in Theorem~\ref{thmb}, we obtain simultaneous level raising at primes $\ell_i$'s satisfying the level raising condition for $\chibar$.
Furthermore, in Theorem~\ref{thmc}, we obtain simultaneous level raising at a set of primes $\ell_i$'s which consists of primes of \emph{both the types} described above.

\begin{rem}
The results obtained in this article are mostly disjoint from those of \cite{D2} and most of them cannot be recovered from \cite[Theorem B]{D2}.
To be precise, the cases of Theorems~\ref{thmb} and \ref{thmc} \emph{not covered} by \cite[Theorem B]{D2} are exactly those for which \emph{one} of the following conditions hold:
\begin{enumerate}
\item $\dim(H^1(G_{\QQ,Np},\chibar)) > 1$,
\item $p \nmid \ell_i+1$ for some $1 \leq i \leq r$.
\end{enumerate}
\end{rem}

Let $\zeta_p$ be a primitive $p$-th root of unity and denote the class group of $\QQ(\zeta_p)$ by $\text{Cl}(\QQ(\zeta_p))$.
Given a character $\chi$ of $\text{Gal}(\QQ(\zeta_p)/\QQ)$, denote by $\text{Cl}(\QQ(\zeta_p))/p\text{Cl}(\QQ(\zeta_p))[\chi]$ the subspace of $\text{Cl}(\QQ(\zeta_p))/p\text{Cl}(\QQ(\zeta_p))$ on which $\text{Gal}(\QQ(\zeta_p)/\QQ)$ acts via character $\chi$.
Denote the $k$-th Bernoulli number by $B_{k}$.
From now on, we will use the notation $p \mid B_k$ (resp. $p \nmid B_k$) to mean that $p$ divides (resp. does not divide) the numerator of $B_k$.
As a consequence of Theorem~\ref{thmb}, we get the following corollaries which establish the conjecture of Billerey--Menares (\cite[Conjecture 3.2]{BM}) in many cases:
\begin{cor}
\label{corf}
Let $\rhob_0 = 1 \oplus \omega_p^{k_0-1}$, where $k_0$ is an even integer such that $2 < k_0 <p-1$. 
Let $k$ be an integer such that $k \equiv k_0 \pmod{p-1}$.
Let $\ell_1,\cdots,\ell_r$ be primes such that $p \mid \ell_i^{k_0-2} - 1$ and $p \nmid \ell_i-1$ for all $1 \leq i \leq r$.
Suppose $\text{Cl}(\QQ(\zeta_p))/p\text{Cl}(\QQ(\zeta_p))[\omega_p^{k_0}]=0$ and $p \mid B_{k_0}$.
Then there exists a newform $f$ of level $\Gamma_0(\prod_{i=1}^{r}\ell_i)$ and weight $k$ such that $\rho_f$ lifts $\rhob_0$.
\end{cor}

\begin{cor}
\label{corb}
Let $\rhob_0 = 1 \oplus \omega_p^{k_0-1}$, where $k_0$ is an even integer such that $2 < k_0 <p-1$. 
Let $k$ be an integer such that $k \equiv k_0 \pmod{p-1}$.
Let $\ell_0$ be a prime such that $p \mid \ell_0^{k_0}-1$ and $p \nmid \ell_0-1$.
Let $\ell_1,\cdots,\ell_r$ be primes such that $p \mid \ell_i^{k_0-2} - 1$ and $p \nmid \ell_i-1$ for all $1 \leq i \leq r$.
Suppose $p \nmid B_{k_0}$.
Then there exists a newform $f$ of level $\Gamma_0(\prod_{i=0}^{r}\ell_i)$ and weight $k$ such that $\rho_f$ lifts $\rhob_0$.
\end{cor}
 
In the next few remarks, we will elaborate a bit more on the cases of the conjecture of Billerey--Menares that are covered by the corollaries above and the cases that remain to be proved.
In what follows, $2 < k_0 < p-1$ is an even integer.
\begin{rem}
Note that Corollary~\ref{corf} proves the cases of the Billerey--Menares conjecture given by its second condition under some additional assumptions.
To be precise, suppose $p \mid B_{k_0}$, $N =\prod_{\i=1}^{r}\ell_i$ and $\ell_i^{k_0-2} \equiv 1 \pmod{p}$ for all $1 \leq i \leq r$.
Then we prove that there exists a newform of level $N$ and weight $k_0$ lifting $1 \oplus \omega_p^{k_0-1}$ if $\text{Cl}(\QQ(\zeta_p))/p\text{Cl}(\QQ(\zeta_p))[\omega_p^{k_0}]=0$ (which is implied by Vandiver's conjecture) and $p \nmid \phi(N)$.
\end{rem}

\begin{rem}
Note that Corollary~\ref{corb} proves the cases of the Billerey--Menares conjecture given by its first condition under some additional assumptions.
To be precise, $N =\prod_{\i=1}^{r}\ell_i$ and $(\ell_i^{k_0-2}-1)(\ell_i^{k_0}-1) \equiv 0 \pmod{p}$ for all $1 \leq i \leq r$.
Then we prove that there exists a newform of level $N$ and weight $k_0$ lifting $1 \oplus \omega_p^{k_0-1}$ if the following conditions hold:
\begin{enumerate}
\item $p \nmid \phi(N)$ and $p \nmid B_{k_0}$,
\item There exists at most one prime $\ell \mid N$ such that $\ell^{k_0} \equiv 1 \pmod{p}$ and $p \nmid \ell^2-1$.
\end{enumerate}
\end{rem}

\begin{rem}
Our results do not prove the Billerey--Menares conjecture for certain classes of levels (even in a single case).
For instance, suppose $N$ is a squarefree number such that $p \nmid N$ and it satisfies one of the following conditions:
\begin{enumerate}
\item $p \mid \phi(N)$,
\item $N=\ell_1\ell_2\ell_3$ such that $ \ell_i^2-1 \not\equiv 0 \pmod{p}$ for $ 1 \leq i \leq 3$, $\ell_i^{k_0} \equiv 1 \pmod{p}$ for $i=1,2$ and $\ell_3^{k_0-2}  \equiv 1 \pmod{p}$.
\end{enumerate}
Then the Billerey--Menares conjecture states that there exists a newform of level $N$ and weight $k_0$ lifting $1 \oplus \omega_p^{k_0-1}$.
We do not prove any results for such levels and these cases are not covered by \cite[Corollary 5.3.5]{D2} as well.
An explicit example of the case of second type is $p=37$, $k_0 =6$, $\ell_1=11$, $\ell_2 = 233$ and $\ell_3 = 43$.
\end{rem}

In \cite{BM}, Billerey and Menares also find a logarithmic lower bound for the number of newforms of weight $k$ and prime level belonging to an explicit set of primes of natural lower density at least $3/4$ (\cite[Theorem 2]{BM}).
Moreover, after assuming their conjecture, they extend this result to an appropriate family $\mathcal{N}_r$ of squarefree integers (\cite[Theorem 4.2]{BM}).
They crucially use \cite[Theorem 2]{LMP} which, roughly speaking, is a result about large factors of $p-1$ for a prime $p$.
As a consequence, their result depends on their conjecture holding for squarefree levels $N = \prod_{i=1}^{r}\ell_i$ such that $p \mid \ell_i-1$ \emph{for all} $1 \leq i \leq r$.

In Corollaries~\ref{corf} and \ref{corb}, we prove \cite[Conjecture 3.2]{BM} for levels $N$ such that $p \nmid \phi(N)$.
However, \cite[Theorem 2]{LMP} is also true if you replace $p_i-1$ by $p_i+1$ in the definition of $\mathcal{A}_{k,c}$ in loc. cit. (see the discussion after the theorem in \cite{LMP}).
Note that, in Corollary~\ref{corb}, we prove \cite[Conjecture 3.2]{BM} for squarefree levels $N = \prod_{i=1}^{r}\ell_i$ such that $p \mid \ell_i + 1$ for all $1 \leq i \leq r$ under the assumption that $p \nmid B_{k_0}$.

Let $\mathcal{N}_r'$ be the set obtained by replacing $p_i-1$ with $p_i+1$ in the definition of $\mathcal{N}_r$ occurring in \cite[Section 4.1]{BM}.
Then, following the proof of \cite[Theorem 4.2]{BM}, we conclude that Corollary~\ref{corb} implies \cite[Theorem 4.2]{BM} for the set $\mathcal{N}_r'$.



We now state our last main result where we obtain level raising by a square of a prime.
\begin{thm}
\label{thme}
Suppose we are in the Set-up~\ref{setup} as above, $\chibar \neq \omega_p$ and $H^1_{\{p\}}(G_{\QQ,Np},\chibar^{-1}) = 0$. 
Let $k>2$ be an integer such that $k \equiv k_0 \pmod{p-1}$. 
Let $\ell$ be a prime such that $p \mid \ell+1$ and $\chibar|_{G_{\QQ_{\ell}}} = \omega_p|_{G_{\QQ_{\ell}}}$.
Let $\ell_1,\cdots,\ell_r$ be primes such that $\ell_i \nmid Np$, $\chibar|_{G_{\QQ_{\ell_i}}} = \omega_p|_{G_{\QQ_{\ell_i}}}$ and $p \nmid \ell_i-1$ for all $1 \leq i \leq r$. Then:
\begin{enumerate}
\item\label{bhaag1} There exists an eigenform $f$ of level $N\ell^2$ and weight $k$ such that $\rho_f$ lifts $\rhob_0$ and $f$ is new at $\ell$.
\item\label{bhaag2} There exists an eigenform $f'$ of level $N\ell^2\prod_{i=1}^{r}\ell_i$ and weight $k$ such that $\rho_{f'}$ lifts $\rhob_0$, $f'$ is new at $\ell$ and $f'$ is new at $\ell_i$ for every $1 \leq i \leq r$.
\end{enumerate}
\end{thm}

If $f$ is the eigenform obtained in Theorem~\ref{thme}, then the newness of $f$ at $\ell$ means that $\ell^2 \mid N'$ where $N'$ is the level of the newform underlying $f$.
\begin{rem}
Note that, an analogue of Theorem~\ref{thme} is \emph{not} proved in \cite{D2}.
Moreover, the results of the type of Part~\eqref{bhaag2} of Theorem~\ref{thme} have not been proved earlier for reducible residual representations to the best of our knowledge.
\end{rem}

As a corollary, we get: 
\begin{cor}
\label{squarecor}
Let $\rhob_0 = 1 \oplus \omega_p^{k_0-1}$, where $k_0$ is an even integer such that $2 < k_0 <p-1$.
Let $k$ be an integer such that $k \equiv k_0 \pmod{p-1}$ and $\ell$ be a prime such that $p \mid \ell+1$.
Let $\ell_1,\cdots,\ell_r$ be primes such that $p \mid \ell_i^{k_0-2} - 1$ and $p \nmid \ell_i-1$ for all $1 \leq i \leq r$.
Suppose $p \nmid B_{k_0}$. Then:
\begin{enumerate}
\item\label{ek} There exists a newform $f$ of level $\Gamma_0(\ell^2)$ and weight $k$ such that $\rho_f$ lifts $\rhob_0$.
\item\label{don} There exists a newform $f'$ of level $\Gamma_0(\ell^2\prod_{i=1}^{r}\ell_i)$ and weight $k$ such that $\rho_{f'}$ lifts $\rhob_0$.
\end{enumerate}
\end{cor}

Note that Part~\eqref{ek} of Corollary~\ref{squarecor} is a partial generalization of Part (a) of \cite[Theorem B]{LW} to higher weights.

\begin{rem}
The proof of \cite[Theorem B]{LW} crucially relies on the result of Mazur \cite{M2} which asserts that there does not exist a newform of prime level $\Gamma_0(\ell)$ and weight $2$ lifting $1 \oplus \omega_p$ if $p \mid \ell+1$.
On the other hand, Billerey--Menares \cite{BM} prove the existence of such a newform when $k >2$.
So the arguments of \cite{LW} do not yield, without any additional inputs, Part~\eqref{ek}  of Corollary~\ref{squarecor}.
\end{rem}

\begin{rem}
Hypotheses of Corollaries~\ref{corf}, \ref{corb} and \ref{squarecor} are satisfied when $p$ is a regular prime.
Moreover, for a fixed $k_0$, the hypotheses of Corollaries~\ref{corb} and \ref{squarecor} are satisfied by all but finitely many primes $p$.
\end{rem}

\subsection{Cyclicity of $H^1_{\{p\}}(G_{\QQ,Np},\chibar^{-1})$}
\label{assumesec}
We will now briefly analyze the assumptions on $H^1_{\{p\}}(G_{\QQ,Np},\chibar^{-1})$ appearing in our main results.
Since we have assumed $\chibar|_{G_{\QQ_p}} \neq \omega_p^{-1}|_{G_{\QQ_p}}$, it follows, from the local Euler characteristic formula, that $\dim(H^1(G_{\QQ_p},\chibar^{-1}|_{G_{\QQ_p}}))=1$.
So we conclude that $\dim(H^1(G_{\QQ,Np},\chibar^{-1})) \leq \dim(H^1_{\{p\}}(G_{\QQ,Np},\chibar^{-1})) + 1$.

We begin by analyzing the vanishing of $H^1_{\{p\}}(G_{\QQ,Np},\chibar^{-1})$ (which appears as a hypothesis in Theorems~\ref{thmc} and \ref{thme}).
Since $\chibar$ is odd, global Euler characteristic formula implies that $\dim(H^1(G_{\QQ,Np},\chibar^{-1})) \geq 1$.
Therefore we see that $$H^1_{\{p\}}(G_{\QQ,Np},\chibar^{-1}) =0 \implies \dim(H^1(G_{\QQ,Np},\chibar^{-1})) =1.$$
We refer the reader to \cite[Section 1.4]{D2} for a brief summary of known results about cyclicity of $H^1(G_{\QQ,Np},\bar\eta)$ for odd characters $\bar\eta$.

Suppose $K_0$ is the fixed field of $\chibar^{-1}$ and $\text{Cl}(K_0)$ is its class group.
Then it is easy to verify, using the Greenberg-Wiles formula (\cite[Theorem 2]{Wa}), that $H^1_{\{p\}}(G_{\QQ,Np},\chibar^{-1}) =0$ if and only if the following conditions hold:
\begin{enumerate}
\item $\chibar^{-1}$-component of $\text{Cl}(K_0)/p\text{Cl}(K_0)$ is zero,
\item $\chibar^{-1}|_{G_{\QQ_{\ell}}} \neq \omega_p|_{G_{\QQ_{\ell}}}$ for all primes $\ell \mid N.$
\end{enumerate}
Thus, when $N=1$ and $\chibar=\omega_p^{k_0-1}$ for an even $2 < k_0 < p$, then Herbrand-Ribet theorem implies that 
$$H^1_{\{p\}}(G_{\QQ,p},\omega_p^{1-k_0}) =0 \iff p \nmid B_{k_0}.$$
 Note that this condition is satisfied when $p$ is a regular prime.

We now move on to the hypothesis $\dim(H^1_{\{p\}}(G_{\QQ,Np},\chibar^{-1}))=1$ appearing in Theorem~\ref{thmb}. 
Note that this means $\dim(H^1(G_{\QQ,Np},\chibar^{-1})) \leq 2$ with equality holding in many cases.
For instance, from the Greenberg-Wiles formula (\cite[Theorem 2]{Wa}), we conclude that $$\dim(H^1(G_{\QQ,Np},\chibar^{-1})) \leq 2 \implies \chibar^{-1}|_{G_{\QQ_{\ell}}} = \omega_p|_{G_{\QQ_{\ell}}} \text{ for at most one prime } \ell \mid N.$$
Moreover, it also implies that if $H^1_{\{p\}}(G_{\QQ,Np},\chibar^{-1}) =0$ and $\ell$ is a prime such that $\ell \nmid Np$ and $\chibar^{-1}|_{G_{\QQ_\ell}} = \omega_p|_{G_{\QQ_\ell}}$, then $\dim(H^1_{\{p\}}(G_{\QQ,N\ell p},\chibar^{-1}))=1$ and $\dim(H^1(G_{\QQ,N\ell p},\chibar^{-1}))=2$.
So this cyclicity hypothesis is weaker than the one appearing in \cite[Theorem B]{D2}.

Suppose $\chibar^{-1}|_{G_{\QQ_{\ell}}} \neq \omega_p|_{G_{\QQ_{\ell}}}$ for all primes $\ell \mid N$.
Then $\dim(H^1_{\{p\}}(G_{\QQ,Np},\chibar^{-1}))=1$ if and only if the $\chibar^{-1}$-component of $\text{Cl}(K_0)/p\text{Cl}(K_0)$ is non-trivial and cyclic.
Thus, if $2 < k_0 < p-1$ is an even integer, $N=1$, $\chibar^{-1} = \omega_p^{1-k_0}$ and $p \mid B_{k_0}$, then Vandiver's conjecture, along with the Herbrand-Ribet theorem, implies that $\dim(H^1_{\{p\}}(G_{\QQ,Np},\chibar^{-1}))=1$ (see \cite[Theorem 22]{BK}).
From \cite[Corollary 3.8]{K}, it follows that for all primes $p > 5$ either $H^1_{\{p\}}(G_{\QQ,p},\omega_p^3) = 0$ or $\dim(H^1_{\{p\}}(G_{\QQ,p},\omega_p^3))=1$.
Note that the cyclicity of $H^1_{\{p\}}(G_{\QQ,Np},\chibar^{-1})$ is also related to the Gorenstein property of certain Hecke algebras (see \cite{K2} and \cite{O}).

\subsection{Sketch of the proofs of the main results}
Note that Theorem~\ref{thmc} (except for the case when $p \mid \ell_0+1$) follows easily from Theorem~\ref{thmb}. 
The Corollaries \ref{corf} and \ref{corb} follow easily from Theorem~\ref{thmb} and Theorem~\ref{thmc}, respectively. 
So we will now give a brief sketch of proof of Theorem~\ref{thmb}.
To prove the theorem, we mainly follow the strategy used in \cite[Sec. 5]{D2} to prove \cite[Theorem B]{D2}. 
However, there are some differences between the two strategies which we will highlight below.

Note that we are assuming $k >2$ and let $M=N\prod_{i=1}^{r}\ell_i$. 
In view of the modularity lifting theorem of Skinner-Wiles (\cite[Theorem A]{SW2}), \cite[Proposition 2]{Ca} and \cite[Lemma 5.1.2]{D2}, our strategy is to construct a lift $\rho : G_{\QQ, Mp} \to \GL_2(\overline{\QQ_p})$ with appropriate determinant of a reducible, non-split representation $\rhob_c : G_{\QQ,Mp} \to \GL_2(\FF)$ with semi-simplification $\rhob_0$ such that
\begin{enumerate}
\item $\rho$ is $p$-ordinary and irreducible,
\item $\rho$ is ramified (and in fact, Steinberg) at $\ell_i$ for every $1 \leq i \leq r$.
\end{enumerate}
We construct this lift from $\rho_c^{\ord}$, the universal ordinary deformation of $\rhob_c$. To do so, observe that it suffices to find the following:
\begin{enumerate}
\item A quotient $R$ of $R^{\ord}_{\rhob_c}$, the universal ordinary deformation ring of $\rhob_c$, such that it is a finite $\ZZ_p$-algebra of Krull dimension $1$ and the determinant of $\rho_c^{\ord}$ has the right shape in it, 
\item A minimal prime $P$ of $R$ such that the corresponding representation over $R/P$ (i.e. the representation obtained from $\rho_c^{\text{ord}}$) is \emph{ramified} and \emph{reducible} at every $\ell_i$.
\end{enumerate}
The choice of the residual representation $\rhob_c$ is an important difference between the strategy of this article and that of \cite{D2}.
In \cite{D2}, we chose the residual representation $\rhob_c : G_{\QQ,Mp} \to \GL_2(\FF)$ to be of the form $ \begin{pmatrix} \chibar_1 & * \\ 0  & \chibar_2\end{pmatrix}$ with $* \neq 0$.
In this article, we consider an extension in the other direction for the residual representation.
To be precise, we take $\rhob_c : G_{\QQ,Mp} \to \GL_2(\FF)$ such that $\rhob_c \simeq  \begin{pmatrix} \chibar_1 & 0 \\ *  & \chibar_2\end{pmatrix}$ with $* \neq 0$ corresponding to a non-zero element $c  \in H^1(G_{\QQ,Np},\chibar^{-1})$.
Since ordinary deformations of $\rhob_c$ play an important role in our argument, we choose $c \in H^1_{\{p\}}(G_{\QQ,Np},\chibar^{-1})$ (which is non-zero by assumption) so that they exist.

Before moving on to the other difference between the strategies, note that very few finiteness results of the type alluded to in the first point above are available in the literature. 
Moreover, such results are not readily available under the hypotheses of Theorem~\ref{thmb}.
However, Pan (\cite{P}) has recently proved a finiteness result for \emph{pseudo-deformation} rings under much less restrictive hypotheses (see \cite[Theorem 5.1.2]{P}).
His result implies that if $R^{\ps}_{\rhob_0}$ is the universal deformation ring of the pseudo-representation associated to $\rhob_0$, then the map $R^{\ps}_{\rhob_0} \to R^{\ord}_{\rhob_c}$ induced by the pseudo-representation associated to universal ordinary deformation of $\rhob_c$ factors through a finite $\ZZ_p\llbracket T \rrbracket$-algebra.
When this map is surjective, his result gives us the desired finiteness statement.
However, this map is not always surjective.

In \cite{D2}, the assumption $\dim(H^1(G_{\QQ,Np},\chibar)) =1$ was used to get a surjection from $R^{\ps}_{\rhob_0}$ to $R(N)_{\rhob_c}$, the universal deformation ring of $\rhob_c$ as a representation of $G_{\QQ,Np}$.
This surjection was in turn used to prove that the map $R^{\ps}_{\rhob_0} \to R^{\ord}_{\rhob_c}$  is surjective (modulo some zero divisors).
Note that, under the assumptions of Theorem~\ref{thmb}, the map $R^{\ps}_{\rhob_0} \to R(N)_{\rhob_c}$ is not necessarily surjective (and in fact, it is not surjective most of the times).
Here we make a key observation that if $\dim(H^1_{\{p\}}(G_{\QQ,Np},\chibar^{-1}))=1$, then the surjectivity of the map $R^{\ps}_{\rhob_0} \to R(N)_{\rhob_c}$  is not needed to establish the surjectivity, modulo some zero divisors, of the map $R^{\ps}_{\rhob_0} \to R^{\ord}_{\rhob_c}$ (see Lemma~\ref{surjlem} and Lemma~\ref{surjlemma}).
This is the other key difference between the strategy adapted here and the strategy of \cite{D2}.

 In order to finish the construction mentioned in the first point above, we first fix a suitable lift $\Upsilon$ of $\det(\rhob_c)$ to $W(\FF)$. 
Let $R^{\ord,\det}_{\rhob_c}$ be the universal ordinary deformation ring of $\rhob_c$ with constant determinant $\Upsilon$.
Note that it is a quotient of $R^{\ord}_{\rhob_c}$. 
We then find a quotient $R'$ of $R^{\ord,\det}_{\rhob_c}$ by some zero divisors such that the map $R^{\ps}_{\rhob_0} \to R'$ obtained by passing to this quotient is surjective (see Lemma~\ref{surjlemma}).
The choice of $\rhob_c$, along with the hypothesis $\dim(H^1_{\{p\}}(G_{\QQ,Np},\chibar^{-1}))=1$, plays a crucial role in finding this quotient.

Now we need to prove that $R'$ has Krull dimension $1$ and has a prime ideal of the type mentioned in the second point above.
In order to do this, we prove a result relating the structures of the universal deformation rings of $\rhob_c$ for the groups $G_{\QQ,Mp}$ and $G_{\QQ,Np}$ (see Proposition~\ref{strprop}).
This result is a partial generalization of a similar result of B\"{o}ckle (\cite[Theorem 4.7]{Bo1}). The hypothesis $\dim(H^1_{\{p\}}(G_{\QQ,Np},\chibar^{-1}))=1$ and Pan's finiteness result play a crucial role in its proof.
This structure theorem then allows us to conclude that $R'$ has Krull dimension $1$ and has a prime ideal having desired properties from which Theorem~\ref{thmb} follows.

To prove Theorem~\ref{thme}, we follow a similar strategy. In this case, we choose a representation $\rhob_c : G_{\QQ, N \ell p} \to \GL_2(\FF)$ corresponding to a non-trivial element of $H^1_{\{p\}}(G_{\QQ,N\ell p},\chibar^{-1})$. Note that $\rhob_c$ is ramified at $\ell$. 
Here, we crucially use the relation between $R_{\rhob_c,\ell}$, the universal deformation ring of $\rhob_c|_{G_{\QQ_\ell}}$, and the universal deformation ring of $\rhob_c$ given by \cite[Theorem 3.1]{Bo3}.
To be precise, we first determine the structure of $R_{\rhob_c,\ell}$ (Lemma~\ref{localstrlem}). We then combine it with \cite[Theorem 3.1]{Bo3} and the arguments used in the proof of Theorem~\ref{thmb} to find a desired quotient $R'$ of $R^{\ord,\det}_{\rhob_c}$.

Note that the combination of Lemma~\ref{localstrlem} and \cite[Theorem 3.1]{Bo3} allows us to find modular lifts of $\rhob_c$ of two types: lifts which are irreducible at $\ell$ and lifts which are Steinberg at $\ell$.
The type of lifts depends on the quotient $R'$ that we take.
Combining this with the arguments used in the proof of Theorem~\ref{thmc}, we conclude that all these lifts are Steinberg at all the other level raising primes (if any).
 Thus the modular lifts which are irreducible at $\ell$ are used to prove Theorem~\ref{thme}, while the modular lifts which are Steinberg at $\ell$ take care of Theorem~\ref{thmc} when $p \mid \ell_0+1$.

\subsection{Organization of the paper}
In \S\ref{strsec}, we study the structure of deformation rings of certain reducible, non-split representations.
To be precise, in \S\ref{defring}, we define the deformation and pseudo-deformation rings which we study and analyze ordinary deformations.
In \S\ref{relationsec}, we explore the relationship between deformation and pseudo-deformation rings to get some finiteness results for deformation rings.
In \S\ref{ramsec}, we introduce increasing of ramification and provide background results necessary to prove the main result of the section.
In \S\ref{mainsec}, we prove the main result of the section which determines the structure of the deformation ring after increasing the ramification.
In \S\ref{proofthm}, we give a proof of Theorem~\ref{thmb}.
In \S\ref{proofthm1}, we give a proof of Theorem~\ref{thmc}.
In \S\ref{proofthm2}, we prove Theorem~\ref{thme}.
In \S\ref{proofcor}, we prove all the corollaries.

\subsection{Notations and Conventions} 
\label{notsec}
 For an integer $M$, denote by $G_{\QQ,Mp}$ the Galois group of the maximal extension of $\QQ$ unramified outside primes dividing $Mp$ and $\infty$ over $\QQ$. Denote by $W(\FF)$ the ring of Witt vectors of $\FF$. 
 Let $\mathcal{C}$ be the category of complete, Noetherian, local $W(\FF)$-algebras with residue field $\FF$.
 For a character $\eta : G_{\QQ,Mp} \to \FF^{\times}$ and an object $R$ of $\mathcal{C}
 $, denote the Teichmuller lift of $\eta$ to $R^{\times}$ by $\widehat\eta$.
 Denote the mod $p$ cyclotomic character of $G_{\QQ,Mp}$ by $\omega_p$ and the $p$-adic cyclotomic character of $G_{\QQ,Mp}$ by $\chi_p$.
 For a prime $\ell$, denote the absolute Galois group of $\QQ_{\ell}$ by $G_{\QQ_{\ell}}$ and denote its inertia subgroup by $I_{\ell}$.
 Let $\frob_\ell$ denote the Frobenius at $\ell$ in $G_{\QQ_{\ell}}/I_{\ell}$.
For every prime $\ell$, fix an embedding $\iota_{\ell}: \overline{\QQ} \to \overline{\QQ}_{\ell}$ which in turn will give us a map $i_{\ell,M} : G_{\QQ_{\ell}} \to G_{\QQ,Mp}$.
If $\rho$ is a representation of $G_{\QQ,Mp}$, then denote by $\rho|_{G_{\QQ_{\ell}}}$ the representation $\rho \circ i_{\ell,M}$ of $G_{\QQ_{\ell}}$ and if $g \in G_{\QQ_{\ell}}$, then denote $\rho|_{G_{\QQ_{\ell}}}(g)$ by $\rho(g)$.
All the representations, pseudo-representations and Galois cohomology groups that we consider in this article are assumed to be continuous unless mentioned otherwise.
Given a representation $\rho$ of $G_{\QQ,Mp}$ over $\FF$, denote by $\dim(H^i(G_{\QQ,Mp},\rho))$, the dimension of $H^i(G_{\QQ,Mp},\rho)$ as a vector space over $\FF$.

\subsection{Acknowledgments}
I would like to thank Nicolas Billerey, Jaclyn Lang and Preston Wake for providing some helpful comments on an earlier draft of the article.
I would also like to thank the anonymous referees for a careful reading of the article and for providing many useful comments and suggestions which helped tremendously in improving the exposition and the main results of this article.
This work was partially supported by a Young Investigator Award from the Infosys Foundation, Bangalore and also by the DST FIST program - 2021 [TPN - 700661]. 

\section{Structure of deformation rings}
\label{strsec}
In this section, we introduce the deformation rings of certain reducible, non-split representations with semi-simplification $\rhob_0$ and study their relationship with universal deformation ring of the pseudo-representation $(\tr(\rhob_0),\det(\rhob_0))$.
As a consequence, we prove various results about their structure which will play a key role in the proofs of the main theorems.
Throughout this section, we assume that we are in the Set-up~\ref{setup}. 
We start by establishing some notation in the next subsection.

\subsection{Deformation rings and ordinary deformations}
\label{defring}
Observe that the hypotheses on $\chibar_1$ and $\chibar_2$ from Set-up~\ref{setup} imply that $\chibar|_{G_{\QQ_p}} \neq 1, \omega_p^{-1}$.
So, there exists a $g_0 \in G_{\QQ_p}$ such that $\chibar_1(g_0) \neq \chibar_2(g_0)$.
Fix such a $g_0 \in G_{\QQ_p}$.
Note that the restriction of $c$ to $\ker(\chibar)$ is a homomorphism $\ker(\chibar) \to \FF$ which we will also denote by $c$.
Fix an $h_0 \in G_{\QQ,Np}$ such that $\chibar(h_0)=1$ (i.e. $h_0 \in \ker(\chibar)$) and $c(h_0) \neq 0$.

For a non-zero element $c \in H^1(G_{\QQ,Np},\chibar^{-1})$, let $\rhob_c : G_{\QQ,Np} \to \GL_2(\FF)$ be the representation such that
\begin{enumerate}
\item $\rhob_c(g) = \begin{pmatrix} \chibar_2(g) & *\\ 0 & \chibar_1(g) \end{pmatrix}$ for all $g \in G_{\QQ,Np}$, where $*$ corresponds to $c$,
\item $\rhob_c(g_0) = \begin{pmatrix} \chibar_2(g_0) & 0\\ 0 & \chibar_1(g_0) \end{pmatrix}$,
\item $\rhob_c(h_0) = \begin{pmatrix} 1 & 1\\ 0 & 1 \end{pmatrix}$.
\end{enumerate}
Note that this means there exists a unique cocycle $\cF \in Z^1(G_{\QQ,Np},\chibar^{-1})$ such that image of $\cF$ in $H^1(G_{\QQ,Np},\chibar^{-1})$ is $c$ and $$\rhob_c(g) = \begin{pmatrix} \chibar_2(g) & \chibar_1(g)\cF(g)\\ 0 & \chibar_1(g) \end{pmatrix}$$ for all $g \in G_{\QQ,Np}$.

Denote by $R(N)_{\rhob_c}$ the universal deformation ring of $\rhob_c$ in $\mathcal{C}$ and let $\rho_N^{\univ} : G_{\QQ,Np} \to \GL_2(R(N)_{\rhob_c})$ be a (representation in the equivalence class giving the) universal deformation of $\rhob_c$. The existence of $R(N)_{\rhob_c}$ and $\rho_N^{\univ}$ follows from \cite{M} and \cite{Ra}.
Note that $(\tr(\rhob_0),\det(\rhob_0)) : G_{\QQ,Np} \to \FF$ is a pseudo-representation (in the sense of Chenevier \cite{C}) of $G_{\QQ,Np}$ of dimension $2$.
Denote by $R^{\ps}(N)_{\rhob_0}$ the universal deformation ring of the pseudo-representation $(\tr(\rhob_0),\det(\rhob_0))$ in $\mathcal{C}$. 
The existence of $R^{\ps}(N)_{\rhob_0}$ follows from \cite{C}.

If $M$ is an integer such that $N \mid M$, then we can also view $\bar\rho_c$ as a representation and $(\tr(\rhob_0),\det(\rhob_0))$ as a pseudo-representation of the group $G_{\QQ,Mp}$.
In this case, denote by $R(M)_{\rhob_c}$ (resp. by $R^{\ps}(M)_{\rhob_0}$) the universal deformation ring of $\rhob_c$ (resp. of $(\tr(\rhob_0),\det(\rhob_0))$) in $\mathcal{C}$ for the group $G_{\QQ,Mp}$ and let $\rho_M^{\univ} : G_{\QQ,Mp} \to \GL_2(R(M)_{\rhob_c})$ be a (representation in the equivalence class giving the) universal deformation of $\rhob_c$ for $G_{\QQ,Mp}$.
Denote by $\mathfrak{m}_M$ the maximal ideal of $R(M)_{\rhob_c}$.

By \cite[Lemma $3.1.1$]{D2}, there exists a $P \in \GL_2(R(M)_{\rhob_c})$ such that $P \equiv Id \pmod{\mathfrak{m}_M}$ and $P\rho_M^{\univ}(g_0)P^{-1} = \begin{pmatrix} a_0 & 0\\ 0 & b_0 \end{pmatrix}$.
So we can choose $\rho_M^{\univ}$ such that $\rho_M^{\univ}(g_0)$ is diagonal.
We make this choice and assume that $\rho_M^{\univ}(g_0)$ is diagonal throughout the article unless mentioned otherwise.

If $R$ is an object of $\mathcal{C}$, $M$ is an integer divisible by $N$ and $\rho : G_{\QQ,Mp} \to \GL_2(R)$ is a deformation of $\rhob_c$, then we say that $\rho$ is an \emph{ordinary} deformation of $\rhob_c$ if there exist characters $\eta_1,\eta_2 : G_{\QQ_p} \to R^{\times}$ such that 
\begin{enumerate}
\item $\eta_i$ is a lift of $\bar\chi_i|_{G_{\QQ_p}}$ for $i=1,2$, 
\item $\eta_2$ is unramified,
\item $\rho|_{G_{\QQ_p}} \simeq \begin{pmatrix} \eta_2 & 0 \\ * & \eta_1 \end{pmatrix}$.
\end{enumerate}

We now make a simple yet crucial observation:
\begin{lem}
\label{diagonallem}
Suppose $c \in H^1_{\{p\}}(G_{\QQ,Np},\chibar^{-1})$ is a non-zero element.
Let $M$ be an integer divisible by $N$ and $k$ be an integer such that $k \equiv k_0 \pmod{p-1}$.
There exist $\alpha, \beta, \delta_k \in R(M)_{\rhob_c}$ such that $\rho_M^{\univ} \pmod{(\alpha,\beta,\delta_k)}$ is an ordinary deformation of $\rhob_c$ with determinant $\epsilon_k\chi_p^{k-1}$, where $\epsilon_k$ is unramified at $p$.
\end{lem}
\begin{proof}
Recall that $\rho_M^{\univ}(g_0) = \begin{pmatrix} a_0 & 0\\ 0 & b_0 \end{pmatrix}$.
Thus we have $a_0 \pmod{\mathfrak{m}_M}=\chibar_2(g_0)$, $b_0 \pmod{\mathfrak{m}_M} =\chibar_1(g_0)$ and hence, $a_0 \not\equiv b_0 \pmod{\mathfrak{m}_M}$.
Since $g_0 \in G_{\QQ_p}$, it follows, from \cite[Lemma $2.4.5$]{Bel}, that there exist ideals $B_p$ and $C_p$ of $R(M)_{\rhob_c}$ such that $$R(M)_{\rhob_c}[\rho_M^{\univ}(G_{\QQ_p})] = \begin{pmatrix} R(M)_{\rhob_c} & B_p \\ C_p & R(M)_{\rhob_c}\end{pmatrix}.$$
Now $c \in H^1_{\{p\}}(G_{\QQ,Np},\chibar^{-1})$ which means $\rhob_c(G_{\QQ_p})$ is an abelian group. Moreover, $\rhob_c(g_0)$ is a non-scalar diagonal matrix.
Hence, $\rhob_c(g)$ is diagonal for all $g \in G_{\QQ_p}$.
So it follows that $B_p, C_p \subset \mathfrak{m}_M$.

Since $\chibar|_{G_{\QQ_p}} \neq 1, \omega_p^{-1}$, by local Euler characteristic formula, it follows that $\dim_{\FF}(H^1(G_{\QQ_p},\chibar^{-1})) =1$. So, by \cite[Theorem $1.5.5$]{BC}, the ideal $B_p$ is generated by at most $1$ element (see also Part $(5)$ of \cite[Lemma $2.4$]{D} and its proof).
 Let $\alpha$ be a generator of the ideal $B_p$ if $B_p \neq (0)$ and $0$ otherwise. Note that $\alpha \in \mathfrak{m}_M$ because $B_p \subset \mathfrak{m}_M$.

 Recall that if $G_{\QQ_p}^{\text{ab}}$ is the maximal abelian quotient of $G_{\QQ_p}$, then $$G_{\QQ_p}^{\text{ab}} \simeq \ZZ/(p-1)\ZZ \times (1+p\ZZ_p) \times \hat\ZZ.$$ Choose a lift $i_p$ of a topological generator of $1+p\ZZ_p$ in $I_p$.
So $\det(\rho_M^{\univ}(i_p)) = 1+\gamma$ for some $\gamma \in \mathfrak{m}_M$. Let $\delta_k = 1 + \gamma - \chi_p^{k-1}(i_p)$. Note that $\delta_k \in \mathfrak{m}_M$.

If $\rho_M^{\univ}(i_p) = \begin{pmatrix} a & b\\ c & d \end{pmatrix}$ then $a=\hat\chibar_2(i_p)(1+\beta)$ for some $\beta \in \mathfrak{m}_M$. 
Now let $R := R(M)_{\rhob_c}/(\alpha,\beta,\delta_k)$ and $\rho := \rho_M^{\univ} \pmod{(\alpha,\beta,\delta_k)}$.
Then it is easy to verify that $\rho : G_{\QQ,Mp} \to \GL_2(R)$ is an ordinary deformation of $\rhob_c$ with determinant $\epsilon_k\chi_p^{k-1}$, where $\epsilon_k$ is unramified at $p$ (see proof of \cite[Lemma $3.1.3$]{D2} for more details).
\end{proof}

\begin{rem}
If $p \nmid \phi(M)$, $k$ is an integer such that $k \equiv k_0 \pmod{p-1}$ and $\alpha$, $\beta$, $\delta_k \in R(M)_{\rhob_c}$ are elements found in Lemma~\ref{diagonallem}, then $R(M)_{\rhob_c}/(\alpha,\beta,\delta_k)$ is isomorphic to the universal ordinary deformation ring $\mathcal{R}$ of $\rhob_c$ with constant determinant $\widehat{\bar\psi}\chi_p^{k-1}$.
Indeed, if $\xi : R(M)_{\rhob_c} \to \mathcal{R}$ is the surjective map induced by the universal ordinary deformation of $\rhob_c$ corresponding to $\mathcal{R}$, then Lemma~\ref{diagonallem} implies that $\ker(\xi) \subset (\alpha,\beta,\delta_k)$.
On the other hand, there exists an $\mathcal{N} \in \GL_2(\mathcal{R})$ such that $\mathcal{N}(\xi \circ \rho^{\univ}_M|_{G_{\QQ_p}})\mathcal{N}^{-1} = \begin{pmatrix} \eta_2 & 0 \\ * & \eta_1\end{pmatrix}$ with $\eta_i$ lifting $\chibar_i$ and $\eta_2$ an unramified character of $G_{\QQ_p}$.
From the description of $\rho^{\univ}_M(g_0)$ given above, it follows that $\xi(a_0) = \eta_2(g_0)$ and $\xi(b_0) = \eta_1(g_0)$.
This implies that $\mathcal{N}$ is a lower triangular matrix and hence, $\xi \circ \rho^{\univ}_M|_{G_{\QQ_p}} = \begin{pmatrix} \eta_2 & 0 \\ * & \eta_1\end{pmatrix}$. 
This allows us to conclude that $\xi(\alpha)=0$.
As $\eta_2$ is an unramified character, it follows, from definition of $\beta$, that $\xi(\beta)=0$.
Since $\det(\xi \circ \rho^{\univ}_M) = \widehat{\bar\psi}\chi_p^{k-1}$, we obtain, using definition of $\delta_k$, that $\xi(\delta_k)=0$.
Thus we get $(\alpha,\beta,\delta_k) \subset \ker(\xi)$ which proves our claim.
\end{rem}
\subsection{Relationship between $R(M)^{\ps}_{\rhob_0}$ and $R(M)_{\rhob_c}$}
\label{relationsec}
We now study the relation between the universal deformation ring $R(M)_{\rhob_c}$ and the universal pseudo-deformation ring $R(M)^{\ps}_{\rhob_0}$.
To do this, we first analyze the reducibility properties of deformations of $\rhob_c$.

\begin{lem}
\label{redlem}
Suppose $c \in H^1(G_{\QQ,Np},\chibar^{-1})$ is a non-zero element.
Let $M$ be an integer divisible by $N$, $R$ be an object in $\mathcal{C}$ with maximal ideal $m_R$ and $\rho : G_{\QQ,Mp} \to \GL_2(R)$ be a deformation of $\rhob_c$.
If there exist characters $\chi_1, \chi_2 : G_{\QQ,Mp} \to R^{\times}$ such that $\chi_i$ is a lift of $\chibar_i$ for $i=1,2$ and $\tr(\rho)=\chi_1+\chi_2$, then there exists a $P \in \GL_2(R)$ such that $P \equiv Id \pmod{m_R}$, $P \rho(g_0) P^{-1}$ is diagonal and $P \rho P^{-1} = \begin{pmatrix} \chi_2 & * \\ 0 & \chi_1\end{pmatrix}$.
\end{lem}
\begin{proof}
The proof is similar to those of \cite[Lemme $1$]{BC1} and \cite[Lemma $3.2$]{D3}.
But we give the details here for the sake of completion.
By \cite[Lemma $3.1.1$]{D2}, we get a $P \in \GL_2(R)$ such that $P \equiv Id \pmod{m_R}$ and $P\rho(g_0)P^{-1} = \begin{pmatrix} a_0 & 0\\ 0 & b_0 \end{pmatrix}$ with $a_0 \pmod{m_R}=\chibar_2(g_0)$ and $b_0 \pmod{m_R} =\chibar_1(g_0)$.
So $a_0$ and $b_0$ are roots of the polynomial $$X^2-\tr(\rho(g_0))X +\det(\rho(g_0)) =  (X-\chi_1(g_0))(X-\chi_2(g_0)).$$
As $a_0 \not\equiv b_0 \pmod{m_R}$, it follows that $a_0=\chi_2(g_0)$ and $b_0 = \chi_1(g_0)$.

Denote the representation $P \rho P^{-1}$ by $\rho'$.
Now suppose $\rho'(g) = \begin{pmatrix} a_g & b_g \\ c_g & d_g \end{pmatrix}$ for $g \in G_{\QQ,Mp}$. Then $a_g+d_g = \chi_2(g)+\chi_1(g)$ and $$\tr(\rho'(g_0g)) = a_0a_g+b_0d_g = \chi_2(g_0g)+\chi_1(g_0g) = \chi_2(g_0)\chi_2(g)+\chi_1(g_0)\chi_1(g) = b_0\chi_2(g)+a_0\chi_1(g).$$
Therefore, we get $(a_0-b_0)d_g = (a_0-b_0)\chi_1(g)$. As $a_0 - b_0 \not\in m_R$, we get that $d_g =\chi_1(g)$ and hence, $a_g =\chi_2(g)$.
So $$\rho'(g) = \begin{pmatrix} \chi_2(g) & b_g \\ c_g & \chi_1(g) \end{pmatrix}$$ for all $g \in G_{\QQ,Mp}$.

As $\rho'$ is a lift of $\rhob_c$, it follows that there exists a $g' \in G_{\QQ,Mp}$ with $b_{g'} \in R^{\times}$.
Thus $\chi_2(g'g) = \chi_2(g')\chi_2(g)+b_{g'}c_g$ for all $g \in G_{\QQ,Mp}$.
Therefore, we get that $c_g=0$ for all $g \in G_{\QQ,Mp}$ which proves the lemma.
\end{proof}

This allows us to find a quotient of the deformation ring of $\rhob_c$ on which the pseudo-deformation ring of $\rhob_0$ surjects:
\begin{lem}
\label{surjlem}
Let $M$ be an integer divisible by $N$ and suppose $\dim(H^1_{\{p\}}(G_{\QQ,Mp},\chibar^{-1})) =1$.
Let $c \in H^1_{\{p\}}(G_{\QQ,Mp},\chibar^{-1})$ be a non-zero element and $\alpha \in R(M)_{\rhob_c}$ be the element found in Lemma~\ref{diagonallem}.
The morphism $\phi : R(M)^{\ps}_{\rhob_0} \to R(M)_{\rhob_c}/(\alpha)$ induced by $(\tr(\tilde\rho),\det(\tilde\rho))$, where $\tilde\rho := \rho_M^{\univ}\pmod{(\alpha)}$, is surjective.
\end{lem}
\begin{proof}
Denote the maximal ideals of $R(M)^{\ps}_{\rhob_0}$ and $R' := R(M)_{\rhob_c}/(\alpha)$ by $\mathfrak{n}_M$ and $\bar{\mathfrak{m}}_M$, respectively.
We claim that to prove the lemma, it is enough to prove that the ideal $I$ of $R'$ generated by $\phi(\mathfrak{n}_M)$ is $\bar{\mathfrak{m}}_M$.
To prove this claim, note that $\phi(\mathfrak{n}_M) \subset \bar{\mathfrak{m}}_M$ and the residue fields of $R(M)^{\ps}_{\rhob_0}$ and $R'$ are the same.
Hence, if $I = \bar{\mathfrak{m}}_M$, then we get that the map induced by $\phi$ on the corresponding graded rings (associated to the filtrations by the powers of respective maximal ideals) is surjective.
Therefore, we conclude, using \cite[Lemma 10.23]{AM}, that $\phi$ is surjective.

We will now prove that $I = \bar{\mathfrak{m}}_M$.
Let $R'' = R'/I$. So $R''$ is an $\FF$-algebra.
Observe that $I = \bar{\mathfrak{m}}_M$ if and only if $R'' \simeq \FF$.
Now $R'' \simeq \mathbb{F}$ if and only if there does not exist a surjective map $f : R'' \to \FF[\epsilon]/(\epsilon^2)$.
Let $f : R'' \to \FF[\epsilon]/(\epsilon^2)$ be a map and let $\rho := f \circ \tilde\rho$.
Since $R''$ is a quotient of $R(M)_{\rhob_c}$, it follows that $f$ is surjective if and only if $\rho \not\simeq \rhob_c \otimes_{\FF} \FF[\epsilon]/(\epsilon^2)$.
Hence, to prove the lemma, it suffices to prove that $\rho \simeq \rhob_c \otimes_{\FF} \FF[\epsilon]/(\epsilon^2)$.

Note that $\tr(\rho)=\tr(\rhob_0)$ and $\det(\rho) = \det(\rhob_0)$.
Then, from Lemma~\ref{redlem}, we know that there exists a $P \in \GL_2(\FF[\epsilon]/(\epsilon^2))$ such that $P \equiv Id \pmod{(\epsilon)}$ $$ P \rho P^{-1}= \begin{pmatrix} \chibar_2 & * \\ 0 & \chibar_1\end{pmatrix} \text{ and }P \rho(g_0) P^{-1} = \begin{pmatrix} \chibar_2(g_0)  & 0 \\ 0 & \chibar_1(g_0)\end{pmatrix}.$$

 From our choice of $\rho_M^{\univ}$, we get that $\rho(g_0)= \begin{pmatrix} \tilde{a_0} & 0\\ 0 & \tilde{b_0} \end{pmatrix}$.
So $\tilde{a_0}$ and $\tilde{b_0}$ are roots of the polynomial  $X^2-\tr(\rhob_0(g_0))X +\det(\rhob_0(g_0)) =  (X-\chibar_1(g_0))(X-\chibar_2(g_0)).$
As $\chibar_1(g_0) \neq \chibar_2(g_0)$, we conclude that $\tilde{a_0} = \chibar_2(g_0)$, $\tilde{b_0} = \chibar_1(g_0)$ and hence, $$\rho(g_0) = \begin{pmatrix}  \chibar_2(g_0)  & 0 \\ 0 & \chibar_1(g_0) \end{pmatrix} = P \rho(g_0) P^{-1}.$$
Since $\chibar_2(g_0) \neq \chibar_1(g_0)$, this implies that $P$ is a diagonal matrix.
So $\rho$ is also of the form $\begin{pmatrix} \chibar_2 & * \\ 0 & \chibar_1\end{pmatrix}$.

This means that there exists a cocycle $\cF' \in Z^1(G_{\QQ,Mp},\chibar^{-1})$ such that 
$$\rho(g) = \begin{pmatrix} \chibar_2(g) &\chibar_1(g)( \cF(g)+\epsilon \cF'(g)) \\ 0  & \chibar_1(g) \end{pmatrix} \text{ for all } g \in G_{\QQ,Mp}.$$
Let $c'$ be the image of $\cF'$ in $H^1(G_{\QQ,Mp},\chibar^{-1})$.



It follows, from the proof of Lemma~\ref{diagonallem} and the definition of $\alpha$, that $\tilde\rho|_{G_{\QQ_p}} = \begin{pmatrix} \eta_2 & 0 \\ * & \eta_1 \end{pmatrix}$ for some characters $\eta_2$ and $\eta_1$ lifting $\bar\chi_2$ and $\bar\chi_1$, respectively.
Since $\rho = f \circ \tilde\rho$, the previous paragraph implies that $$\rho(g) = \begin{pmatrix} \chibar_2(g) & 0 \\ 0 & \chibar_1(g) \end{pmatrix} \text{ for all }g \in G_{\QQ_p}.$$
Hence, $c' \in H^1_{\{p\}}(G_{\QQ,Mp},\chibar^{-1})$. Since $\dim(H^1_{\{p\}}(G_{\QQ,Mp},\chibar^{-1}))=1$ and $\rho(g_0)$ is diagonal, it follows that there exist a $x_0 \in \FF$ such that $\cF'=x_0\cF$.
Therefore, conjugating $\rho$ by $\begin{pmatrix} 1 - x_0\epsilon & 0\\ 0 & 1\end{pmatrix}$, we get $\rhob_c \otimes_{\FF} \FF[\epsilon]/(\epsilon^2)$. This proves the lemma.
\end{proof}

As a consequence, we get:

\begin{lem}
\label{dimlem}
Let $M$ be an integer divisible by $N$. Suppose $\dim(H^1_{\{p\}}(G_{\QQ,Mp},\chibar^{-1})) =1$ and $p \nmid \phi(M)$.
Let $c \in H^1_{\{p\}}(G_{\QQ,Mp},\chibar^{-1})$ be a non-zero element, $k$ be an integer such that $k \equiv k_0 \pmod{p-1}$ and $\alpha, \beta, \delta_k \in R(M)_{\rhob_c}$ be the elements found in Lemma~\ref{diagonallem}.
Then 
\begin{enumerate}
\item $R(M)_{\rhob_c}$ is a local complete intersection ring of Krull dimension $4$,
\item $R(M)_{\rhob_c}/(\alpha,\beta,\delta_k)$ is a finite $W(\FF)$-algebra and a local complete intersection ring of Krull dimension $1$.
\end{enumerate}
\end{lem}
\begin{proof}
Let $R := R(M)_{\rhob_c}/(\alpha,\beta,\delta_k)$ and $\rho : G_{\QQ,Mp} \to \GL_2(R)$ be the deformation of $\rhob_c$ obtained by composing $\rho_M^{\univ}$ with the natural surjective map $R(M)_{\rhob_c} \to R$.
Combining Lemma~\ref{diagonallem} and the hypothesis $p \nmid \phi(M)$, it follows that the pseudo-representation $$(\tr(\rho \otimes \widehat\chibar_2^{-1}), \det(\rho \otimes \widehat\chibar_2^{-1})) : G_{\QQ,Mp} \to R$$ is a deformation of the pseudo-representation $(1+\chibar,\chibar) : G_{\QQ,Mp} \to \FF$ such that
\begin{enumerate}
\item $\tr(\rho \otimes \widehat\chibar_2^{-1}|_{G_{\QQ_p}}) = \eta_1 + \eta_2$, where $\eta_1,\eta_2 : G_{\QQ_p} \to R^{\times}$ are characters and $\eta_2$ is an unramified lift of $1$,
\item $\det(\rho \otimes \widehat\chibar_2^{-1}) = \epsilon\chi_p^{k-1}$, where $\epsilon : G_{\QQ,Mp} \to W(\FF)^{\times}$ is a character of finite order.
\end{enumerate}
Note that the hypothesis $p \nmid \phi(M)$ is needed to conclude that the character $\epsilon$ takes values in $W(\FF)^{\times}$. Otherwise, we can only conclude that it is a finite character taking values in $R^{\times}$.

Let $\mathbb{S}$ be the universal deformation ring of the pseudo-representation $(1 + \chibar,\chibar): G_{\QQ,Mp} \to \FF$ in $\mathcal{C}$ and $\psi : \mathbb{S} \to R$ be the map induced by $(\tr(\rho \otimes \widehat\chibar_2^{-1}), \det(\rho \otimes \widehat\chibar_2^{-1}))$.
Let $S^{\circ}$ be the universal deformation ring of the pseudo-representation $(1+\chibar,\chibar): G_{\QQ,Mp} \to \FF$ which represents the functor from $\mathcal{C}$ to the category of sets sending an object $\mathcal{R}$ of $\mathcal{C}$ to the the set of pseudo-representations $(t,d) : G_{\QQ,Mp} \to \mathcal{R}$ lifting $(1+\chibar,\chibar)$ such that 
 $t|_{G_{\QQ_p}}$ is reducible and $d= \epsilon\chi_p^{k-1}$.
Let $(T,D) : G_{\QQ,Mp} \to S^{\circ}$ be the corresponding universal pseudo-representation.
From the previous paragraph, it follows that the map $\psi$ factors through $S^{\circ}$. Let $\psi' : S^{\circ} \to R$ be the morphism induced by $\psi$.

Note that, $T|_{G_{\QQ_p}} = \Phi_1 + \Phi_2$, where $\Phi_1, \Phi_2 : G_{\QQ_p} \to (S^{\circ})^{\times}$ are characters lifting $\chibar|_{G_{\QQ_p}}$ and $1$, respectively.
As $\Phi_2$ is a character of $G_{\QQ_p}$ lifting $1$, $\Phi_2|_{I_p}$ factors through a quotient of $I_p$ which is isomorphic to $1 +p\ZZ_p$.
The completed group ring of $1+p\ZZ_p$ over $W(\FF)$ is isomorphic to $W(\FF)\llbracket T \rrbracket$.
So the character $\Phi_2|_{I_p}$ induces a map $\kappa : W(\FF)\llbracket T \rrbracket \to S^{\circ}$.
It follows, from \cite[Theorem 5.1.2]{P}, that $S^{\circ}$ is a finite $W(\FF)\llbracket T \rrbracket$-algebra under the map $\kappa$.

Recall that the character $\eta_2 : G_{\QQ_p} \to R^{\times}$ lifts $1$. As $\chibar|_{G_{\QQ_p}} \neq 1$, we get, from \cite[Proposition 1.5.1]{BC}, that $\psi' \circ \Phi_2 = \eta_2$.
Moreover, $\eta_2$ is an unramified character of $G_{\QQ_p}$.
Hence, we obtain that $\psi' \circ \kappa (T)=0$.
From the finiteness of $S^{\circ}$ over $W(\FF)\llbracket T \rrbracket$ under $\kappa$, we conclude that $\psi'$ factors through the quotient $\mathbb{S}' := S^{\circ}/(\kappa(T))$ of $S^{\circ}$ which is a finite $W(\FF)$-algebra.
From Lemma~\ref{surjlem}, the map $R(M)^{\ps}_{\rhob_0} \to R$ induced by $(\tr(\rho),\det(\rho))$ is surjective.
Hence, the map $\psi$ is also surjective as it is induced by a twist of $(\tr(\rho),\det(\rho))$.

Note that the map $\psi$ is surjective and it factors through $\mathbb{S}'$ which is finite over $W(\FF)$.
Therefore, we conclude that $R$ is a finite $W(\FF)$-algebra which means its Krull dimension is at most $1$.
So we get, from \cite[Theorem $10.2$]{E}, that $R(M)_{\rhob_c}$ has Krull dimension at most $4$.
But we know, by combining \cite[Theorem $2.4$]{Bo2} and the global Euler characteristic formula, that $$R(M)_{\rhob_c} \simeq \dfrac{W(\FF) \llbracket X_1,\cdots,X_n \rrbracket}{I},$$ where the minimal number of generators of $I$ is at most $n-3$.
If the minimal number of generators of $I$ is $n'$, then \cite[Theorem $10.2$]{E} implies that the Krull dimension of $R(M)_{\rhob_c}$ is at least $1+n-n' \geq 1+n-(n-3)=4$.
Hence, we conclude that $R(M)_{\rhob_c}$ has Krull dimension $4$ and the minimal number of generators of $I$ is $n-3$ which also yields that $R(M)_{\rhob_c}$ is a local complete intersection ring.
Applying \cite[Theorem $10.2$]{E} again, we get that $R$ is a local complete intersection ring of Krull dimension $1$.
\end{proof}

\subsection{Increasing the ramification}
\label{ramsec}
We will now focus on increasing ramification at some specific types of primes. 
For a prime $\ell \neq p$, fix a lift $g_{\ell} \in G_{\QQ_{\ell}}$ of $\frob_{\ell}$ and a lift $i_{\ell} \in I_{\ell}$ of the topological generator of the $\ZZ_p$-quotient of the tame inertia group at $\ell$.

\begin{lem}
\label{ramlem}
Let $c \in H^1(G_{\QQ,Np},\chibar^{-1})$ be a non-zero element.
Let $\ell_1,\cdots,\ell_r$ be primes such that $\ell_i \nmid Np$, $p \nmid \ell_i-1$ and $\chibar|_{G_{\QQ_{\ell_i}}} = \omega_p|_{G_{\QQ_{\ell_i}}}$ for all $1 \leq i \leq r$. Let $M=N\prod_{i=1}^{r}\ell_i$.
Then for every $1 \leq i \leq r$, the universal deformation $\rho_{M}^{\univ} : G_{\QQ,Mp} \to \GL_2(R(M)_{\rhob_c})$ of $\rhob_c$ is tamely ramified at $\ell_i$.
Moreover, there exists a matrix $P_i \in \GL_2(R(M)_{\rhob_c})$ such that $P_i \pmod{\mathfrak{m}_M} = \begin{pmatrix} 1 & * \\ 0 & 1\end{pmatrix}$, $P_i(\rho_{M}^{\univ}(g_{\ell_i}))P_i^{-1} = \begin{pmatrix} \psi_{2,i} & 0 \\ 0 & \psi_{1,i} \end{pmatrix}$ and 
\begin{enumerate}
 \item $P_i(\rho_{M}^{\univ}(i_{\ell_i}))P_i^{-1} = \begin{pmatrix} 1 & 0 \\ w_i & 1 \end{pmatrix},$
 if  $p \nmid \ell_i+1$,
\item $P_i(\rho_{M}^{\univ}(i_{\ell_i}))P_i^{-1} = \begin{pmatrix} \sqrt{1+u_iv_i} & u_i \\ v_i & \sqrt{1+u_iv_i} \end{pmatrix},$  if  $p \mid \ell_i+1.$
\end{enumerate}
\end{lem}
\begin{proof}
Note that for every $1 \leq i \leq r$, there exists a $\tilde Q_i = \begin{pmatrix} 1 & b_i \\ 0 & 1 \end{pmatrix} \in \GL_2(\FF)$ such that $\tilde Q_i \rhob_c(g_{\ell_i}) \tilde Q_i^{-1}$ is diagonal with \emph{distinct} entries on diagonal.
Let $\hat b_i \in R(M)_{\rhob_c}$ be the Teichmuller lift of $b_i$ and let $Q_i = \begin{pmatrix} 1 & \hat b_i \\ 0 & 1 \end{pmatrix} \in \GL_2(R(M)_{\rhob_c})$.
By \cite[Lemma $3.1.1$]{D2}, for every $1 \leq i \leq r$, there exists a $P'_i \in \GL_2(R(M)_{\rhob_c}$ such that $P'_i \equiv Id \pmod{\mathfrak{m}_M}$ and $$P'_i(Q_i\rho_M^{\univ}(g_{\ell_i}){Q_i}^{-1}){P'_i}^{-1} = \begin{pmatrix} \psi_{2,i} & 0 \\ 0 & \psi_{1,i} \end{pmatrix}.$$ 
The lemma now follows directly from \cite[Lemma $4.9$]{Bo1} (see also \cite[Lemma 6, 7]{Bos} as well).
\end{proof}

We now focus on determining structure of $R(N\prod_{i=1}^{r}\ell_i)_{\rhob_c}$ in terms of structure of $R(N)_{\rhob_c}$ (in the spirit of \cite[Theorem $4.7$]{Bo1}). We begin with proving some results which will play a key role in determining this structure.

\begin{lem}
\label{surjlemma}
Suppose $\dim(H^1_{\{p\}}(G_{\QQ,Np},\chibar^{-1})) =1$ and $p \nmid \phi(N)$.
Let $\ell_1,\cdots,\ell_r$ be primes such that $\ell_i \nmid Np$, $p \nmid \ell_i-1$ and $\chibar|_{G_{\QQ_{\ell_i}}} = \omega_p|_{G_{\QQ_{\ell_i}}}$ for all $1 \leq i \leq r$. Let $M=N\prod_{i=1}^{r}\ell_i$.
Let $c \in H^1_{\{p\}}(G_{\QQ,Np},\chibar^{-1})$ be a non-zero element, $k$ be an integer such that $k \equiv k_0 \pmod{p-1}$ and $\alpha, \beta, \delta_k \in R(M)_{\rhob_c}$ be the elements found in Lemma~\ref{diagonallem}.
Let $S$ be the subset of $\{\ell_{1},\cdots,\ell_{r}\}$ consisting of all primes which are $-1 \pmod{p}$.
\begin{enumerate}
\item If $S=\emptyset$, then the morphism $$\phi' : R(M)^{\ps}_{\rhob_0} \to R(M)_{\rhob_c}/(\alpha,\beta,\delta_k)$$ induced by $(\tr(\rho'),\det(\rho'))$, where $\rho' := \rho_M^{\univ}\pmod{(\alpha,\beta,\delta_k)}$, is surjective and $R(M)_{\rhob_c}/(\alpha,\beta,\delta_k)$ is a finite $W(\FF)$-algebra.
\item Suppose $\emptyset \neq S = \{\ell_{i_1},\cdots,\ell_{i_s}\}$ and let $u_{i_1},\cdots,u_{i_s} \in R(M)_{\rhob_c}$ be the elements found in Lemma ~\ref{ramlem}.
Then the morphism $$\phi' : R(M)^{\ps}_{\rhob_0} \to R(M)_{\rhob_c}/(\alpha,\beta,\delta_k,u_{i_1},\cdots,u_{i_s})$$ induced by $(\tr(\rho'),\det(\rho'))$, where $\rho' := \rho_M^{\univ}\pmod{(\alpha,\beta,\delta_k,u_{i_1},\cdots,u_{i_s})}$, is surjective and $R(M)_{\rhob_c}/(\alpha,\beta,\delta_k,u_{i_1},\cdots,u_{i_s})$ is a finite $W(\FF)$-algebra.
\end{enumerate}
\end{lem}
\begin{proof}
First suppose $S=\emptyset$ which means $p \nmid \ell_i+1$ for all $1 \leq i \leq r$.
Since $\chibar|_{G_{\QQ_{\ell_i}}} = \omega_p|_{G_{\QQ_{\ell_i}}}$ for all $1 \leq i \leq r$, this implies that $\chibar^{-1}|_{G_{\QQ_{\ell_i}}} \neq \omega_p|_{G_{\QQ_{\ell_i}}}$ for all $1 \leq i \leq r$.
Hence, by Greenberg--Wiles formula (\cite[Theorem $2$]{Wa}), it follows that $$0 \leq \dim(H^1_{\{p\}}(G_{\QQ,Mp},\chibar^{-1})) - \dim(H^1_{\{p\}}(G_{\QQ,Np},\chibar^{-1})) \leq \sum_{i=1}^{r}\dim(H^0(G_{\QQ_{\ell_i}},\omega_p\chibar|_{G_{\QQ_{\ell_i}}}))=0,$$ which means $\dim(H^1_{\{p\}}(G_{\QQ,Mp},\chibar^{-1})) = \dim(H^1_{\{p\}}(G_{\QQ,Np},\chibar^{-1}))  =1$.
 So, the surjectivity of $\phi'$ in this case follows directly from Lemma~\ref{surjlem}.
Since $p \nmid \phi(M)$, the finiteness of $R(M)_{\rhob_c}/(\alpha,\beta,\delta_k)$ follows from Lemma~\ref{dimlem}.
This completes the proof of the Lemma in the case $S=\emptyset$.

Now suppose $S \neq \emptyset$. 
Without loss of generality, assume $S=\{\ell_1,\cdots,\ell_s\}$.
Let $J = (\alpha,\beta,\delta_k,u_1,\cdots,u_s)$.
Let $f : R(M)_{\rhob_c}/J \to \FF[\epsilon]/(\epsilon^2)$ be a map and let $\rho := f \circ \rho'$.
To prove the surjectivity of $\phi'$ in this case, it suffices to prove that if $\tr(\rho)=\tr(\rhob_0)$ and $\det(\rho) = \det(\rhob_0)$, then $\rho \simeq \rhob_c \otimes_{\FF} \FF[\epsilon]/(\epsilon^2)$ (see the proof of Lemma~\ref{surjlem} for more details).

Suppose $\tr(\rho)=\tr(\rhob_0)$ and $\det(\rho) = \det(\rhob_0)$.
Since $\tr(\rho)=\tr(\rhob_0)$, it follows that for all $g \in G_{\QQ,Mp}$, $\tr(\rho'(gi_{\ell_j})) - \tr(\rho'(g)) \in \ker(f)$ for all $1 \leq j \leq s$.
For every $1 \leq j \leq s$, let $\bar{v_j} := v_j \pmod{J}$ and $\bar{P_j} := P_j \pmod{J}$. Here $v_j$'s are the elements of $R(M)_{\rhob_c}$ and $P_j$'s are the matrices found in Lemma~\ref{ramlem}. 
This means $$\bar{P_j}\rho'(i_{\ell_j})\bar{P_j}^{-1}=\begin{pmatrix} 1 & 0 \\ \bar{v_j} & 1\end{pmatrix}, \text{ for every } 1 \leq j \leq s.$$

Let $h \in G_{\QQ,Mp}$ such that $\rhob_c(h) = \begin{pmatrix} 1 & x \\ 0 & 1 \end{pmatrix}$ for some $x \in \FF^{\times}$.
Since $P_j \pmod{\mathfrak{m}_M} = \begin{pmatrix} 1 & * \\0 & 1\end{pmatrix}$ for all $1 \leq j \leq s$, it follows that $f(\bar{P_j})\rho(h)f(\bar{P_j})^{-1} \pmod{\epsilon} = \begin{pmatrix} 1 & x \\ 0 & 1 \end{pmatrix}$ for every $1 \leq j \leq s$.
Thus we get, from Lemma~\ref{ramlem}, that for every $1 \leq j \leq s$, $$f(\tr(\rho'(h i_{\ell_j}) ) - \tr(\rho'(h))) = f(\tr(\bar{P_j}\rho'(h i_{\ell_j})\bar{P_j}^{-1} ) - \tr(\bar{P_j}\rho'(h)\bar{P_j}^{-1})) = f(w_j \bar{v_j})$$ for some $w_j \in (R(M)_{\rhob_c}/J)^{\times}$.

Therefore, for all $1 \leq j \leq s$, $\bar{v_j} \in \ker(f)$ and hence, $\rho$ is unramified at $\ell_j$ for all $1 \leq j \leq s$. Thus, $\rho$ factors through $G_{\QQ,M'p}$, where $M'=\frac{M}{\prod_{i=1}^{s}\ell_i}$.
So if $\ell \mid \frac{M'}{N}$, then $p \nmid \ell+1$.
From the proof of the first part of the lemma, we get that $\dim(H^1_{\{p\}}(G_{\QQ,M'p},\chibar^{-1}))  =1$.
Now $\rho$ is a $\FF[\epsilon]/(\epsilon^2)$-valued representation of $G_{\QQ,M'p}$ such that $\tr(\rho)=\tr(\rhob_0)$ and $\det(\rho) = \det(\rhob_0)$.
Therefore, it follows, from the proof of Lemma~\ref{surjlem}, that $\rho \simeq \rhob_c \otimes_{\FF} \FF[\epsilon]/(\epsilon^2)$ and hence, $\phi'$ is surjective.

Following the proof of Lemma~\ref{dimlem}, we conclude that $\phi'$ factors through a quotient of $R(M)^{\ps}_{\rhob_0}$ which is a finite $W(\FF)$-algebra. So surjectivity of $\phi'$ implies that $R(M)_{\rhob_c}/J$ is a finite $W(\FF)$-algebra.
This finishes the proof of the lemma.
\end{proof}

\begin{prop}
\label{oneprimelem}
Suppose $\dim(H^1_{\{p\}}(G_{\QQ,Np},\chibar^{-1})) =1$ and let $c \in H^1_{\{p\}}(G_{\QQ,Np},\chibar^{-1})$ be a non-zero element.
Let $\ell$ be a prime such that $\chibar|_{G_{\QQ_{\ell}}} = \omega_p|_{G_{\QQ_{\ell}}}$ and $p \nmid \ell - 1$.
Then the ring $R(N\ell)_{\rhob_c}$ has Krull dimension $4$.
\end{prop}
\begin{proof}
Suppose $p \nmid \ell^2-1$. Since $\chibar|_{G_{\QQ_{\ell}}} = \omega_p|_{G_{\QQ_{\ell}}}$, this implies that $\chibar^{-1}|_{G_{\QQ_{\ell}}} \neq \omega_p|_{G_{\QQ_{\ell}}}$.
Hence, by Greenberg--Wiles formula (\cite[Theorem $2$]{Wa}), it follows that $$0 \leq \dim(H^1_{\{p\}}(G_{\QQ,N\ell p},\chibar^{-1})) - \dim(H^1_{\{p\}}(G_{\QQ,Np},\chibar^{-1})) \leq \dim(H^0(G_{\QQ_{\ell}},\omega_p\chibar|_{G_{\QQ_{\ell}}}))=0,$$ which means $\dim(H^1_{\{p\}}(G_{\QQ,N\ell p},\chibar^{-1})) = \dim(H^1_{\{p\}}(G_{\QQ,Np},\chibar^{-1}))  =1$.
 So, the proposition in this case follows directly from Lemma~\ref{dimlem}.


Now suppose $p \mid \ell + 1$. 
Note that, by combining \cite[Theorem $2.4$]{Bo2}, the global Euler characteristic formula and \cite[Theorem $10.2$]{E}, we get that the Krull dimension of $R(N\ell)_{\rhob_c}$ is at least $4$. Hence, it suffices to prove that if $P$ is a minimal prime ideal of $R(N\ell)_{\rhob_c}$, then the Krull dimension of $R(N\ell)_{\rhob_c}/P$ is at most $4$.

By Lemma~\ref{ramlem}, there exists a $\mathcal{P} \in \GL_2(R(N\ell)_{\rhob_c})$ such that $\mathcal{P} \equiv \begin{pmatrix} 1 & * \\ 0 & 1\end{pmatrix} \pmod{\mathfrak{m}_{N\ell}}$, $$\mathcal{P}\rho_{N\ell}^{\univ}(g_\ell)\mathcal{P}^{-1} = \begin{pmatrix} \psi_{2,\ell} & 0 \\ 0  & \psi_{1,\ell} \end{pmatrix} \text{ and } \mathcal{P}\rho_{N\ell}^{\univ}(i_\ell)\mathcal{P}^{-1} = \begin{pmatrix}  \sqrt{1+uv} & u \\ v & \sqrt{1+uv} \end{pmatrix}.$$
Let $\rho := \mathcal{P}\rho_{N\ell}^{\univ}\mathcal{P}^{-1}$.

Let $\psi_{\ell} = \psi_{2,\ell}\psi_{1,\ell}^{-1}$ and $f_{\ell}, h_{\ell} \in R(N\ell)_{\rhob_c}$ be such that $\rho(i_{\ell})^{\ell} = \begin{pmatrix} f_{\ell}  & uh_{\ell} \\ vh_{\ell} &  f_{\ell} \end{pmatrix}$.
Therefore, the relation $\rho(g_{\ell}i_{\ell}g_{\ell}^{-1}) = \rho(i_{\ell})^{\ell}$ gives us $u(h_{\ell}-\psi_{\ell}) =0$ and $v(h_{\ell}-\psi_{\ell}^{-1})=0$ (see proof of \cite[Lemma 4.9]{Bo1} for more details).

Let $k > 2$ be an integer such that $k \equiv k_0 \pmod{p-1}$ and $\alpha, \beta, \delta_k \in R(N\ell)_{\rhob_c}$ be the elements found in Lemma~\ref{diagonallem}.

Let $P$ be a minimal prime of $R(N\ell)_{\rhob_c}$. So $P$ contains either $u$ or $h_{\ell}-\psi_{\ell}$.
If $P$ contains $u$, then Lemma~\ref{surjlemma} implies that $R(N\ell)_{\rhob_c}/(\alpha,\beta,\delta_k,P)$ is finite over $W(\FF)$.
Thus, its Krull dimension is at most $1$. Hence, \cite[Theorem $10.2$]{E} implies that the Krull dimension of $R(N\ell)_{\rhob_c}/P$ is at most $4$.

Suppose $P$ does not contain $u$. 
So $h_{\ell}-\psi_{\ell} \in P$. 
Now Lemma~\ref{surjlemma} implies that $R:=R(N\ell)_{\rhob_c}/(\alpha,\beta,\delta_k,u,P)$ is finite over $W(\FF)$. So its Krull dimension is at most $1$.
Note that $P$ contains either $v$ or $h_{\ell} - \psi_{\ell}^{-1}$. 

Suppose $v \in P$ and $R$ has Krull dimension $1$. Let $Q$ be a minimal prime of $R$. Then $R/Q$ is an integral domain which is a finite algebra over $W(\FF)$.
Hence, it can be identified with a subring of $\overline{\QQ}_p$.
Fix an inclusion $R/Q \to \overline{\QQ}_p$ and let $\tau' := \rho_{N\ell}^{\univ} \pmod{(\alpha,\beta,\delta_k,u,P)}$.
Composing $\tau'\pmod{Q}$ with this inclusion, we get a representation $\tau : G_{\QQ,N\ell p} \to \GL_2(\overline{\QQ_p})$.

As $R(N)_{\rhob_c} \simeq R(N\ell)_{\rhob_c}/(u,v)$, it follows that $\tau$ is unramified at $\ell$.
Now Lemma~\ref{diagonallem} implies that $\tau|_{G_{\QQ_p}} = \begin{pmatrix} \eta_2 & 0 \\ * & \eta_1 \end{pmatrix}$ with $\eta_2$ an unramified character of $G_{\QQ_p}$ and $\det(\tau)=\epsilon_k\chi_p^{k-1}$, where $\epsilon_k$ is a character which is unramified at $p$.
We now claim that $\tau$ is irreducible.

If $\tau$ is reducible, then by combining Hensel's lemma and the proof of Lemma~\ref{redlem}, we get that there exist two characters $\psi_1,\psi_2 : G_{\QQ,N\ell p} \to R^{\times}$ such that $\tr(\rho) = \psi_1+\psi_2$.
Let $\mathfrak{m}$ be the maximal ideal of $R$ and for $i = 1,2$, let $\bar\psi_i = \psi_i \pmod{\mathfrak{m}}$.
Since $\bar\psi_1 + \bar\psi_2 = \chibar_1+\chibar_2$, Brauer-Nesbitt theorem implies that $\{\bar\psi_1,\bar\psi_2\} = \{\chibar_1,\chibar_2\}$.

Hence, $\psi_1$ and $\psi_2$ are lifts of characters $\chibar_1$ and $\chibar_2$.
Without loss of generality, suppose $\psi_i$ is a lift of $\chibar_i$.
Since $\rho(g_0)$ is diagonal (by our choice of $\rho_M^{\univ}$), it follows, from Lemma~\ref{redlem}, that $\rho = \begin{pmatrix} \psi_2 & * \\ 0 & \psi_1\end{pmatrix}$.
Therefore, $\psi_2|_{G_{\QQ_p}}=\eta_2$ and hence, $\psi_2$ is unramified at $p$.
As $p \nmid \phi(N\ell)$, it follows that $\psi_2 = \hat\chibar_2$ and $\psi_1$ and $\psi_2$ are unramified at $\ell$.

By Hensel's lemma, we conclude that $\psi_{2,\ell} =\psi_2(\frob_{\ell})$ and $\psi_{1,\ell} =\psi_1(\frob_{\ell})$.
Observe that $h_{\ell} \equiv \ell \pmod{uv}$ and $u,v, h_{\ell}-\psi_{\ell} \in (\alpha,\beta,\delta_k,u,P)$.
Therefore, it follows that $\ell\psi_1(\frob_{\ell}) = \psi_2(\frob_{\ell}) = \hat\chibar_2(\frob_{\ell})$.
Hence, $\det(\tau(\frob_{\ell})) = \ell^{-1}\hat\chibar_2(\frob_{\ell})^2$.
 But this gives us a contradiction as $k>2$ and $\det(\tau)=\epsilon_k\chi_p^{k-1}$.

Thus, we conclude that $\tau$ is irreducible.
Hence, \cite[Theorem A]{SW2} implies that $\tau$ is the $p$-adic Galois representation $\rho_f$ attached to an eigenform $f$ of tame level $N$.
From the previous paragraph, we get that $\rho_f|_{G_{\QQ_{\ell}}} \simeq \chi\oplus \chi\chi_p$.
Since $\ell \nmid N$, \cite[Lemma $5.1.1$]{D2} gives a contradiction. Hence, the Krull dimension of $R$ is $0$ if $v \in P$. So \cite[Theorem $10.2$]{E} implies that the Krull dimension of $R(N\ell)_{\rhob_c}/P$ is at most $4$.

Now suppose $v \not\in P$ which means $h_{\ell} - \psi_{\ell}^{-1} \in P$. Since $h_{\ell} - \psi_{\ell} \in P$, this means that $\psi_{\ell}^2-1 \in P$ and hence, $\psi_{\ell} + 1 \in P$. 
As $h_{\ell} \equiv \ell \pmod{uv}$, we have $\ell + 1 \in (u,P)$.
Recall that $R$ is finite over $W(\FF)$. So its Krull dimension is $0$ (as $\ell+1 \in (u,P)$).
Therefore, we conclude, using \cite[Theorem $10.2$]{E}, that the Krull dimension of $R(N\ell)_{\rhob_c}/P$ is at most $4$.
This proves the proposition.
\end{proof}

\subsection{Relationship between $R(M)_{\rhob_c}$ and $R(N)_{\rhob_c}$}
\label{mainsec}
Suppose $\dim(H^1_{\{p\}}(G_{\QQ,Np},\chibar^{-1})) =1$ and let $c \in H^1_{\{p\}}(G_{\QQ,Np},\chibar^{-1})$ be a non-zero element.
Let $\ell_1,\cdots,\ell_r$ be primes not dividing $Np$ such that for every $1 \leq i \leq r$, $p \nmid \ell_i-1$ and $\chibar|_{G_{\QQ_{\ell_i}}} = \omega_p|_{G_{\QQ_{\ell_i}}}$.
We will now prove the main result of this section which describes the structure of $R(N\prod_{i=1}^{r}\ell_i)_{\rhob_c}$ in terms of the structure of $R(N)_{\rhob_c}$ (in the spirit of \cite[Theorem 4.7]{Bo1}).

Let $M = N\prod_{i=1}^{r} \ell_i$.
Using the natural surjective map $G_{\QQ,Mp} \to G_{\QQ,N p}$, we can view $\rho^{\univ}_{N}$ as a deformation of $\rhob_c$ to $R(N)_{\rhob_c}$ for the group $G_{\QQ,Mp}$.
This induces a surjective map $\Phi_{M,N} : R(M)_{\rhob_c} \to R(N)_{\rhob_c}$.

Note that when we move from $R(N)_{\rhob_c}$ to $R(M)_{\rhob_c}$, we get some additional variables and relations.
Now $\rho^{\univ}_N(I_{\ell_i})$ is trivial for all $1 \leq i \leq r$ and $\ker(\Phi_{M,N})$ is generated by the entries of the matrices $\{\rho^{\univ}_M(I_{\ell_i}) - Id \mid 1 \leq i \leq r\}$.
Thus the additional variables arise from the images $\rho^{\univ}_M(I_{\ell_i})$ of the inertia groups at $\ell_i$'s.
Recall $\rho^{\univ}_M$ is tamely ramified at every $\ell_i$.
So the additional relations come from the relation between the tame inertia group and Frobenius at every $\ell_i$.
We will now give a reinterpretation of these additional variables and relations.

For every $1 \leq i \leq r$, the deformation $\rho^{\univ}_M|_{G_{\QQ_{\ell_i}}}$ of $\rhob_c|_{G_{\QQ_{\ell_i}}}$ gives a map $R_{\rhob_c,\ell_i} \to R(M)_{\rhob_c}$. 
Here $R_{\rhob_c,\ell_i}$ is the \emph{versal} deformation ring of $\rhob_c|_{G_{\QQ_{\ell_i}}}$ in $\mathcal{C}$ for every $1 \leq i \leq r$.
So the additional variables and relations occurring in $R(M)_{\rhob_c}$ arise from the images of the versal deformation rings $R_{\rhob_c,\ell_i}$ under the maps $R_{\rhob_c,\ell_i} \to R(M)_{\rhob_c}$ given above (see \cite[Theorem 3.1]{Bo3} for more details).
In a special case, we will make the relationship between the local (uni)versal deformation rings and $R(M)_{\rhob_c}$ more explicit (see Proposition~\ref{localstrprop} and Proposition~\ref{localramprop}) and it will be crucially used in the proof of Theorem~\ref{thme}.
We will now give a more precise description of these additional variables and relations.

Before proceeding further, we establish some more notation.
Let $s$ be the number of primes in the set $\{\ell_1,\cdots,\ell_r\}$ which are $-1 \pmod{p}$.
If $s \neq 0$, we assume, without loss of generality, that $p \mid \ell_j+1$ for all $1 \leq j \leq s$.
Let $n$ be the dimension of the tangent space of $R(N)_{\rhob_c}/(p)$.
We now define a power series ring $R_0$ in the following way:
\begin{enumerate}
\item If $s=0$, then define $R_0 := W(\FF) \llbracket X_1,\cdots,X_n,W_1,\cdots,W_r \rrbracket,$
\item If $0 < s <r$, then define $$R_0 := W(\FF) \llbracket X_1,\cdots,X_n,U_1,\cdots, U_s, V_1,\cdots, V_s, W_{s+1},\cdots,W_{r}\rrbracket,$$
\item If $s =r$, then define $R_0 := W(\FF) \llbracket X_1,\cdots,X_n,U_1,\cdots, U_r, V_1,\cdots, V_r \rrbracket$.
\end{enumerate}

 Recall that, by Lemma~\ref{ramlem}, we know that for every $1 \leq i \leq r$, there exists a $P_i \in \GL_2(R(M)_{\rhob_c})$ such that $P_i(\rho_M^{\univ}(g_{\ell_i}))P_i^{-1} = \begin{pmatrix} \psi_{2,i} & 0 \\ 0 & \psi_{1,i} \end{pmatrix}$ and
\begin{enumerate}
 \item $P_i(\rho_M^{\univ}(i_{\ell_i}))P_i^{-1} = \begin{pmatrix} 1 & 0 \\ w_i & 1 \end{pmatrix}$ if $p \nmid \ell_i+1$,
\item $P_i(\rho_M^{\univ}(i_{\ell_i}))P_i^{-1} = \begin{pmatrix} \sqrt{1+u_iv_i} & u_i \\ v_i & \sqrt{1+u_iv_i} \end{pmatrix}$ if $p \mid \ell_i+1$.
\end{enumerate}
If $p \mid \ell_j+1$, then let $g_{\ell_j}, h_{\ell_j} \in R(M)_{\rhob_c}$ be such that $$\rho_M^{univ}(i_{\ell_j})^{\ell_j} = \begin{pmatrix} g_{\ell_j}  & u_jh_{\ell_j} \\ v_jh_{\ell_j} &  g_{\ell_j} \end{pmatrix},$$ and define $\psi_j := \psi_{2,j}\psi_{1,j}^{-1}$.

Under the notation above, we have:

\begin{lem}
\label{powerlem}
There exists a surjective homomorphism $\mathcal{F} : R_0 \to R(M)_{\rhob_c}$ such that
\begin{enumerate}
\item If $s=0$, then $\mathcal{F}(W_i)=w_i$ for all $1 \leq i \leq r$,
\item If $0 < s <r$, then $\mathcal{F}(U_j)=u_j$ and $\mathcal{F}(V_j)=v_j$  for all $1 \leq j \leq s$ and $\mathcal{F}(W_i)=w_{i}$ for all $s+1 \leq i \leq r$,
\item If $s =r$, then $\mathcal{F}(U_j)=u_j$ and $\mathcal{F}(V_j)=v_j$  for all $1 \leq j \leq r$.
\end{enumerate}
\end{lem}

\begin{proof}
We prove the lemma by induction. If $r=1$, then the lemma follows directly from \cite[Theorem 4.7]{Bo1}.
Assume the lemma is true for $r=m$. Now suppose $r=m+1$ and let $M'= \frac{M}{\ell_{m+1}}$.
Note that if $p \nmid \ell_{m+1}+1$, then $R(M)_{\rhob_c}/(w_{m+1}) \simeq R(M')_{\rhob_c}$ and if $p \mid \ell_{m+1}+1$, then $R(M)_{\rhob_c}/(u_{m+1},v_{m+1}) \simeq R(M')_{\rhob_c}$.
Now the lemma follows by combining the induction hypothesis and \cite[Theorem 4.7]{Bo1}.
\end{proof}

Now we are ready to state the main result of this section (we keep the notation as above):

\begin{prop}
\label{strprop}
Let $\mathcal{F}$ be the morphism obtained in Lemma~\ref{powerlem} and let $I_0:=\ker(\mathcal{F})$. Then:
\begin{enumerate}
\item If $s=0$, then there exist $f_1,\cdots, f_{n-3},g_1,\cdots,g_r \in R_0$ such that $\mathcal{F}(g_i) = \psi_{1,i}-\ell_i\psi_{2,i}$ for all $1 \leq i \leq r$ and $I_0$ is generated by the set $$\{f_1,\cdots,f_{n-3},W_1g_1,\cdots,W_rg_r\},$$
\item If $0 <s < r$, then there exist $$f_1,\cdots, f_{n-3},h_1,\cdots,h_s,h'_1,\cdots,h'_s,g_{s+1},\cdots,g_{r} \in R_0$$ such that $\mathcal{F}(h_j) = h_{\ell_j} - \psi_j \text{ and } \mathcal{F}(h'_j) = h_{\ell_j} - \psi_j^{-1}$ for all $1 \leq j \leq s$ and $\mathcal{F}(g_i) = \psi_{1,i}-\ell_{i}\psi_{2,i}$ for all $s+1 \leq i \leq r$ and $I_0$ is generated by the set $$\{f_1,\cdots,f_{n-3},W_{s+1}g_{s+1},\cdots,W_{r}g_{r},U_1h_1,\cdots, U_sh_s,V_1h'_1,\cdots,V_sh'_s\},$$
\item If $s=r$, then there exist $f_1,\cdots, f_{n-3},h_1,\cdots,h_r,h'_1,\cdots,h'_r \in R_0$ such that $\mathcal{F}(h_j) = h_{\ell_j} - \psi_j \text{ and }\mathcal{F}(h'_j) = h_{\ell_j} - \psi_j^{-1}$ for all $1 \leq j \leq r$ and $I_0$ is generated by the set $$\{f_1,\cdots,f_{n-3},U_1h_1,\cdots, U_rh_r,V_1h'_1,\cdots,V_rh'_r\}.$$
\end{enumerate}
In each of these cases, $$R(N)_{\rhob_c} \simeq W(\FF) \llbracket X_1,\cdots,X_n \rrbracket/(\bar{f_1},\cdots,\overline{f_{n-3}}),$$ where $\bar{f_i}$ is the image of $f_i$ modulo the ideal generated by $W_i$'s, $U_j$'s and $V_j$'s.
\end{prop}

\begin{proof}
The proof is similar to that of \cite[Proposition 5.3.1]{D2}.
We will prove the proposition by using induction on $r$. For $r=1$, the proposition follows by combining Lemma~\ref{ramlem}, \cite[Theorem 2.4]{Bo2} and \cite[Theorem 4.7]{Bo1}.
Assume the proposition is true for $r=m$.
Now suppose $r=m+1$.

First assume $0 \leq s <r$. So according to our convention, $p \nmid \ell_{m+1}-1$.
By Lemma~\ref{powerlem}, we know that the map $\mathcal{F}$ is surjective.
Hence, \cite[Theorem 4.7]{Bo1} and the induction hypothesis imply that
\begin{enumerate}
\item If $s=0$, then there exist $$f_1,\cdots, f_{n-3},g_1,\cdots,g_{m+1},F_1,\cdots,F_{m} \in R_0$$ such that $\mathcal{F}(g_i) = \psi_{1,i}-\ell_i\psi_{2,i}$ for all $1 \leq i \leq m+1$, $F_1,\cdots,F_{m}\in (W_{m+1})$ and $I_0$ is generated by the set $$S := \{f_1,\cdots,f_{n-3},W_1g_1+F_1,W_2g_2+F_2,\cdots,W_mg_m+F_m,W_{m+1}g_{m+1}\}.$$
\item If $s \neq 0$, then there exist $$f_1,\cdots, f_{n-3},h_1,\cdots,h_s,h'_1,\cdots,h'_s,g_{s+1},\cdots,g_{m+1} \in R_0$$ and $G_1,\cdots,G_s,H_1,\cdots,H_s,F_{s+1},\cdots,F_{m} \in R_0$ such that $\mathcal{F}(h_j) = h_{\ell_j} - \psi_j \text{ and } \mathcal{F}(h'_j) = h_{\ell_j} - \psi_j^{-1}$ for all $1 \leq j \leq s$, $\mathcal{F}(g_i) = \psi_{1,i}-\ell_i\psi_{2,i}$ for all $s+1 \leq i \leq m+1$, $G_1,\cdots,G_s,H_1,\cdots,H_s,F_{s+1},\cdots,F_{m}\in (W_{m+1})$ and $I_0$ is generated by the set $\{f_1,\cdots,f_{n-3},W_{s+1}g_{s+1}+F_{s+1},\cdots,W_{m}g_{m}+F_{m},W_{m+1}g_{m+1}\} \cup \{U_1h_1+G_1,\cdots, U_sh_s+G_s,V_1h'_1+H_1,\cdots,V_sh'_s+H_s\},$
\end{enumerate}
Note that $\{W_{s+1}g_{s+1},\cdots,W_{m}g_{m}\} \subset \ker(\mathcal{F}).$
For $s+1 \leq i \leq m$, there is an $F'_i \in R_0$ and an element $F''_i$ of the ideal generated by the set $S \setminus \{W_ig_i+F_i\}$ such that $W_ig_i = F'_i(W_ig_i+F_i) + F''_i$.
Suppose $F'_i$ is not a unit. Then we get that $W_ig_i$ is in the ideal generated by the set $(S \setminus \{W_ig_i+F_i\}) \cup \{F_i\}$. 
As $W_ig_i \in \ker(\mathcal{F})$, $F_i \in \ker(\mathcal{F})$.
Hence, it follows that $\ker(\mathcal{F})$ is generated by the set $(S \setminus \{W_ig_i+F_i\}) \cup \{F_i\}$.

Now $F_i \in (W_{m+1})$ and $R(N\ell_i)_{\rhob_c} \simeq R(M)_{\rhob_c}/J_i$, where $J_i$ is the ideal generated by the set $ \{w_j\}_{s+1 \leq j \leq m+1, j \neq i} \cup \{u_j,v_j\}_{1 \leq j \leq s}$ if $s \neq 0$ and by the set $ \{w_j\}_{1 \leq j \leq m+1, j \neq i}$ if $s=0$.
Hence, we get, using \cite[Theorem 10.2]{E}, that $R(N\ell_i)_{\rhob_c}$ has Krull dimension at least $5$.
But Proposition~\ref{oneprimelem} gives a contradiction to this. Therefore, $F'_i$ is a unit for all $s+1 \leq i \leq m$.

This means that $\ker(\mathcal{F})$ is generated by the set $$S':=(S\setminus \{W_{s+1}g_{s+1}+F_{s+1},\cdots,W_{m}g_{m}+F_{m}\}) \cup \{W_{s+1}g_{s+1},\cdots,W_{m}g_{m}\}$$ which proves the proposition in $s=0$ case.

Now suppose $s \neq 0$. Note that $\{U_1h_1,\cdots, U_sh_s,V_1h'_1,\cdots,V_sh'_s\} \subset I_0$.
For $1 \leq j \leq s$, there are elements $G'_j, H'_j \in R_0$, an element $G''_j$ of the ideal generated by the set $S' \setminus \{U_jh_j+G_j\}$ and an element $H''_j$ of the ideal generated by the set $S' \setminus \{V_jh'_j+H_j\}$ such that $U_jh_j = G'_j(U_jh_j+G_j) + G''_j$ and $V_jh'_j = H'_j(V_jh'_j+H_j) + H''_j$.
Note that $R(N\ell_j)_{\rhob_c} \simeq R(M)_{\rhob_c}/I_j$, where $I_j$ is the ideal generated by the set $ \{w_i\}_{s+1 \leq i \leq m+1} \cup \{u_i,v_i\}_{1 \leq i \leq s, i \neq j}$.
Hence, if either $G'_j$ or $H'_j$ is not a unit, then, by applying the same logic as above, we get that $R(N\ell_j)_{\rhob_c}$ has Krull dimension at least $5$.
But Proposition~\ref{oneprimelem} gives a contradiction to this.
Therefore, we get that $G'_j$ and $H'_j$ are units for all $1 \leq j \leq s$.

This means that $\ker(\mathcal{F})$ is generated by the set $$(S'\setminus \{U_1h_1+G_1,\cdots,U_{s}h_{s}+G_{s},V_1h'_1+H_1,\cdots,V_{s}h'_{s}+H_{s}\}) \cup \{U_1h_1,\cdots,U_{s}h_{s},V_1h'_1,\cdots,V_{s}h'_{s}\}$$ which proves the proposition in the case $0 < s <r$.

Now assume $s=r$.
As seen before, \cite[Theorem 4.7]{Bo1} and the induction hypothesis imply that there exist $$G_1,\cdots,G_m,H_1,\cdots H_m \in (U_{m+1},V_{m+1})$$ such that $\ker(\mathcal{F})$ is generated by $$S := \{f_1,\cdots,f_{n-3},U_1h_1+G_1,\cdots, U_mh_m+G_m,V_1h'_1+H_1,\cdots,V_mh'_m+H_m, U_{m+1}h_{m+1}, V_{m+1}h'_{m+1}\}.$$
Note that $$\{U_1h_1,\cdots, U_mh_m,V_1h'_1,\cdots,V_mh'_m\} \subset \ker(\mathcal{F}).$$

For $1 \leq j \leq m$, there are elements $G'_j, H'_j \in R_0$, an element $G''_j$ of the ideal generated by the set $S \setminus \{U_jh_j+G_j\}$ and an element $H''_j$ of the ideal generated by the set $S \setminus \{V_jh'_j+H_j\}$ such that $U_jh_j = G'_j(U_jh_j+G_j) + G''_j$ and $V_jh'_j = H'_j(V_jh'_j+H_j) + H''_j$.
Note that $R(N\ell_j)_{\rhob_c} \simeq R(M)_{\rhob_c}/I_j$, where $I_j$ is the ideal generated by the set $ \{u_i,v_i\}_{1 \leq i \leq m+1, i \neq j}$.
Hence, if either $G'_j$ or $H'_j$ is not a unit, then, by applying the same logic as above, we get that $R(N\ell_j)_{\rhob_c}$ has Krull dimension at least $5$.
But Proposition~\ref{oneprimelem} gives a contradiction to this.
Therefore, we get that $G'_j$ and $H'_j$ are units for all $1 \leq j \leq m$.

This means that $\ker(\mathcal{F})$ is generated by the set $$(S\setminus \{U_1h_1+G_1,\cdots,U_{m}h_{m}+G_{m},V_1h'_1+H_1,\cdots,V_{m}h'_{m}+H_{m}\}) \cup \{U_1h_1,\cdots,U_{m}h_{m},V_1h'_1,\cdots,V_{m}h'_{m}\}$$ which proves the proposition in the remaining case.
\end{proof}

We will now prove some results comparing the deformation rings $R(N)_{\rhob_c}$ and $R(M)_{\rhob_c}$ with local deformation rings. These results will be crucially used in the proof of Theorem~\ref{thme}.
For the rest of the section, assume $H^1_{\{p\}}(G_{\QQ,Np},\chibar^{-1}) = 0$.
Let $\ell$ be a prime such that $p \mid \ell+1$ and $\chibar|_{G_{\QQ_{\ell}}} = \omega_p|_{G_{\QQ_{\ell}}}$.

It follows from Greenberg--Wiles formula (\cite[Theorem 2]{Wa}) that $$\dim(H^1_{\{p\}}(G_{\QQ,N\ell p},\chibar^{-1})) = \dim(H^1_{\{p\}}(G_{\QQ,Np},\chibar^{-1})) + \dim(H^0(G_{\QQ_{\ell}},\chibar\omega_p|_{G_{\QQ_{\ell}}})) = 0+1=1.$$
Let $c \in H^1_{\{p\}}(G_{\QQ,N\ell p},\chibar^{-1})$ be a non-zero element. 
Let $\rhob_c : G_{\QQ,N\ell p} \to \GL_2(\FF)$ be the representation corresponding to $c$ as chosen in \S\ref{defring}.

As $H^1_{\{p\}}(G_{\QQ,Np},\chibar^{-1})=0$, it follows that $c$ is ramified at $\ell$.
Let $R_{\rhob_c,\ell}$ be the universal deformation ring for the representation $\rhob_c|_{G_{\QQ_{\ell}}} : G_{\QQ_{\ell}} \to \GL_2(\FF)$ in $\mathcal{C}$.
Note that $\chibar|_{G_{\QQ_{\ell}}} = \omega_p|_{G_{\QQ_{\ell}}} \neq 1$ as $p \nmid \ell -1$ and $c$ is ramified at $\ell$.
Since the only $G_{\QQ_\ell}$-endomorphisms of $\rhob_c|_{G_{\QQ_\ell}}$ are multiplication by scalars, the existence of $R_{\rhob_c,\ell}$ follows from \cite{M} and \cite{Ra}.

For a representation $\rho : G \to \GL_2(\FF)$, denote by $\ad(\rho)$ the representation on $M_2(\FF)$ in which the action of every $g \in G$ on $M_2(\FF)$ is given by conjugation by $\rho(g)$.
Denote by $\ad^0(\rho)$ the subrepresentation of $\ad(\rho)$ consisting of matrices with trace $0$.
Note that $\ad(\rhob_c|_{G_{\QQ_{\ell}}}) = \ad(\rhob_c)|_{G_{\QQ_{\ell}}}$ and $\ad^0(\rhob_c|_{G_{\QQ_{\ell}}}) = \ad^0(\rhob_c)|_{G_{\QQ_{\ell}}}$.
By abuse of notation, we will denote $\chi_p|_{G_{\QQ_{\ell}}}$ and $\omega_p|_{G_{\QQ_{\ell}}}$ by $\chi_p$ and $\omega_p$, respectively.
We begin by analyzing the structure of $R_{\rhob_c,\ell}$.

\begin{lem}
\label{tangentlem}
The dimension of the tangent space of $R_{\rhob_c,\ell}/(p)$ is $2$.
\end{lem}
\begin{proof}
The dimension of the  tangent space of $R_{\rhob_c,\ell}/(p)$ is $\dim(H^1(G_{\QQ_\ell},\ad(\rhob_c|_{G_{\QQ_{\ell}}})))$ (see \cite[Theorem 2.4]{Bo2}).
However, $\ad(\rhob_c|_{G_{\QQ_\ell}}) = \ad^0(\rhob_c|_{G_{\QQ_\ell}}) \oplus 1$ and $\dim(H^1(G_{\QQ_\ell}, 1))=1$ since $p \nmid \ell-1$.
So it suffices to prove $\dim(H^1(G_{\QQ_\ell},\ad^0(\rhob_c|_{G_{\QQ_{\ell}}}))) = 1$.

Observe that the subspace $V$ of upper triangular matrices with trace $0$ forms a $G_{\QQ_\ell}$-subrepresentation of $\ad^0(\rhob_c|_{G_{\QQ_{\ell}}})$. 
It is easy to verify that $V$ is isomorphic to $\rhob'_{c} := (\rhob_c \otimes \chibar_1^{-1})|_{G_{\QQ_{\ell}}}$.

Note that $c|_{G_{\QQ_{\ell}}}$, the restriction of the global Galois cohomology class $c$ to $G_{\QQ_{\ell}}$, gives a non-zero element of $H^1(G_{\QQ_{\ell}}, \omega_p)$.
Let $\mathfrak{h} \in Z^1(G_{\QQ_{\ell}},\omega_p)$ be a cocycle such that its image in $H^1(G_{\QQ_{\ell}},\omega_p)$ is $c|_{G_{\QQ_{\ell}}}$.
An element $\sigma$ of $H^1(G_{\QQ_{\ell}}, \rhob'_c)$ gives a representation $\rhob_\sigma : G_{\QQ_{\ell}} \to \GL_3(\FF)$ such that 
$$\rhob_\sigma(g) = \begin{pmatrix} \omega_p(g) & \mathfrak{h}(g) & F(g)\\ 0 & 1 & b(g) \\ 0 & 0 & 1\end{pmatrix} \text{ for all } g \in G_{\QQ_{\ell}}.$$
Moreover, $\sigma \neq 0$ if and only if $\rhob_\sigma \not\simeq \rhob'_c \oplus 1$.

Note that $b \in H^1(G_{\QQ_{\ell}},1)$ and a representation of the form given above exists if and only if the coboundary of $-F : G_{\QQ_{\ell}} \to \FF$ is $c|_{G_{\QQ_{\ell}}} \cup b$, the cup product of $c|_{G_{\QQ_{\ell}}}$ and $b$.
As $c|_{G_{\QQ_{\ell}}} \neq 0$ and $\chibar^{-1}|_{G_{\QQ_\ell}} = \omega_p$, we get, from the local Tate duality, that $c|_{G_{\QQ_{\ell}}} \cup b =0$ if and only if $b=0$. 

Now suppose $b=0$ in the representation $\rhob_\sigma$ given above. Then we get that $F \in Z^1(G_{\QQ_{\ell}},\omega_p)$.
Since $\dim(H^1(G_{\QQ_{\ell}},\omega_p))=1$, it follows that there exists a $\lambda \in \FF$ and a coboundary $\mathfrak{h}' \in B^1(G_{\QQ_{\ell}},\omega_p)$ such that
$$\rhob_\sigma(g) = \begin{pmatrix} \omega_p(g) & \mathfrak{h}(g) & \lambda \mathfrak{h}(g) +\mathfrak{h}'(g)\\ 0 & 1 & 0 \\ 0 & 0 & 1\end{pmatrix} \text{ for all } g \in G_{\QQ_{\ell}}.$$
 Now a simple calculation shows that $\rhob_\sigma \simeq \rhob'_c \oplus 1$.
Thus, from the analysis given above, we conclude that $H^1(G_{\QQ_{\ell}},V)= H^1(G_{\QQ_{\ell}}, \rhob'_c) = 0$.

We have the following exact sequence of $G_{\QQ_\ell}$-representations:
$$0 \to V \to \ad^0(\rhob_c|_{G_{\QQ_{\ell}}}) \to \omega_p\to 0.$$
As $H^1(G_{\QQ_{\ell}},V) =0$, local Euler characteristic formula implies that $H^2(G_{\QQ_{\ell}},V) =0$.
Hence, we get the following exact sequence of Galois cohomology groups:
$$0 \to H^1(G_{\QQ_{\ell}}, V) \to H^1(G_{\QQ_\ell},\ad^0(\rhob_c|_{G_{\QQ_{\ell}}})) \to H^1(G_{\QQ_{\ell}},\omega_p) \to 0.$$
Since $\dim(H^1(G_{\QQ_{\ell}},\omega_p)) =1$ and $H^1(G_{\QQ_{\ell}}, V)=0$, we get that $\dim(H^1(G_{\QQ_\ell},\ad^0(\rhob_c|_{G_{\QQ_{\ell}}}))) = 1$ which proves the lemma.
\end{proof}

Let $\rho_{\ell} : G_{\QQ_{\ell}} \to \GL_2(R_{\rhob_c,\ell})$ be a (representation in the equivalence class giving the) universal deformation of $\rhob_c|_{G_{\QQ_{\ell}}}$ and let $\mathfrak{m}_\ell$ be the maximal ideal of $R_{\rhob_c,\ell}$.
We will now use the notation established before Lemma~\ref{ramlem}.
As both $\chibar_1$ and $\chibar_2$ are unramified at $\ell$ and $\rhob_c$ is ramified at $\ell$, it follows that $\rhob_c(I_\ell)$ is a non-trivial $p$-group.
Therefore, $\rho_{\ell}(I_{\ell})$ is a pro-$p$ group.
This means  $\rho_{\ell}$ is tamely ramified. Hence, $\rho_{\ell}(I_{\ell})$ is topologically generated by $\rho_{\ell}(i_{\ell})$ and $\rho_{\ell}(G_{\QQ_\ell})$ is topologically generated by $\rho_{\ell}(i_{\ell})$ and $\rho_{\ell}(g_{\ell})$.

From the proof of Lemma~\ref{ramlem}, it follows that there exists a matrix $\mathcal{P} \in \GL_2(R_{\rhob_c,\ell})$ such that $\mathcal{P} \pmod{\mathfrak{m}_\ell} = \begin{pmatrix} 1 & *\\ 0 & 1 \end{pmatrix}$ and $\mathcal{P}\rho_{\ell}(g_{\ell})\mathcal{P}^{-1}= \begin{pmatrix} \phi_2 & 0\\ 0 & \phi_1 \end{pmatrix}$.
Let $\phi = \phi_1\phi_2^{-1}$.
Suppose $$\mathcal{P}\rho_{\ell}(i_{\ell})\mathcal{P}^{-1}= \begin{pmatrix} 1+x & y\\ z & 1+w \end{pmatrix}.$$ Note that $x,z,w \in \mathfrak{m}_\ell$ and $y \in R_{\rhob_c,\ell}^{\times}$.
Let $x_{\ell}$, $w_{\ell}$, $f_{\ell}$ and $f'_{\ell}$ be elements of $R_{\rhob_c,\ell}$ such that 
$$\mathcal{P}\rho_{\ell}(i_{\ell})^{\ell}\mathcal{P}^{-1}= \begin{pmatrix} 1+x_{\ell} & f'_{\ell}y\\ f_{\ell}z & 1+w_{\ell} \end{pmatrix}.$$

\begin{lem}
\label{maxlem}
The maximal ideal $\mathfrak{m}_\ell$ of $R_{\rhob_c,\ell}$ is generated by the set $\{p,z,\phi_1-\widehat{\bar\chi_1}(\text{Frob}_{\ell})\}$.
\end{lem}
\begin{proof}
Let $R = R_{\rhob_c,\ell}/(p,z,\phi_1-\widehat{\bar\chi_1}(\text{Frob}_{\ell}))$ and let $\pi : \GL_2(R_{\rhob_c,\ell}) \to \GL_2(R)$ be the natural surjective map induced by the quotient map $R_{\rhob_c,\ell} \to R$.
Let $\bar{\mathcal{P}} = \pi(\mathcal{P})$ and $\rhob_{\ell} = \pi \circ \rho_{\ell} : G_{\QQ_{\ell}} \to \GL_2(R)$.
For $r \in R_{\rhob_c,\ell}$, denote its image under the natural surjective map $R_{\rhob_c,\ell} \to R$ by $\bar{r}$.

Note that, $\bar{\mathcal{P}}\rhob_{\ell}(g_{\ell})\bar{\mathcal{P}}^{-1} = \begin{pmatrix} \bar{\phi_2} & 0\\ 0 & {\chibar_1}(\text{Frob}_{\ell}) \end{pmatrix}$ and $\bar{\mathcal{P}}\rhob_{\ell}(i_{\ell})\bar{\mathcal{P}}^{-1} = \begin{pmatrix} 1+\bar{x} & \bar{y} \\ 0 & 1+\bar{w}\end{pmatrix}$.
As $\rhob_{\ell}(G_{\QQ_{\ell}})$ is topologically generated by $\rhob_{\ell}(g_{\ell})$ and $\rhob_{\ell}(i_{\ell})$, it follows that $\rhob_{\ell} \simeq \begin{pmatrix} \psi_2 & * \\ 0 & \psi_1\end{pmatrix}$, where $ \psi_i : G_{\QQ_{\ell}} \to R^{\times}$ is a character lifting $\chibar_i|_{G_{\QQ_{\ell}}}$ for $i=1,2$.
This isomorphism is given by conjugation by $\bar{\mathcal{P}}$.

As $p \nmid \ell-1$, $\psi_1(i_{\ell}) = \psi_2(i_{\ell}) = 1$. 
On the other hand, $\psi_2(i_{\ell}) = 1+\bar{x}$ and $\psi_1(i_{\ell}) = 1+\bar{w}$. So we get $\bar{x} = \bar{w}=0$.
Note that $\psi_1(g_{\ell}) = {\chibar_1}(\text{Frob}_{\ell})$ and $\psi_2(g_\ell) = \bar{\phi_2}$.
Hence, it follows that $\psi_1 = \chibar_1|_{G_{\QQ_{\ell}}}$.
As $\bar{y} \in R^{\times}$, the relation $g_{\ell}i_{\ell}g_{\ell}^{-1} = i_{\ell}^{\ell}$ implies that $\bar{\phi_2} -\ell\chibar_1(\text{Frob}_{\ell}) =0$.
Therefore, we get that $\bar{\phi_2} = \chibar_2(\text{Frob}_{\ell})$ as $R$ is an $\FF$-algebra.
Thus, we conclude that $\psi_2 = \chibar_2|_{G_{\QQ_{\ell}}}$.

Let $f : R \to \FF[\epsilon]/(\epsilon^2)$ be a morphism and let $\rho = f \circ \rhob_{\ell}$.
Suppose $f$ is surjective. As $R$ is a quotient of $R_{\rhob_c,\ell}$, it follows that $\rho$ is a non-trivial deformation of $\rhob_c|_{G_{\QQ_{\ell}}}$ i.e. $\rho \not\simeq \rhob_c|_{G_{\QQ_{\ell}}} \otimes_{\FF} \FF[\epsilon]/(\epsilon^2)$.
We obtain, from the previous paragraphs, that $\rho \simeq \begin{pmatrix} \chibar_2|_{G_{\QQ_\ell}} & * \\ 0 & \chibar_1|_{G_{\QQ_\ell}}\end{pmatrix}$.

Since $\chibar^{-1}|_{G_{\QQ_{\ell}}} =\omega_p$ and $\dim(H^1(G_{\QQ_{\ell}},\omega_p))=1$, we conclude, using the arguments used in the proof of Lemma~\ref{surjlem}, that $\rho \simeq \rhob_c|_{G_{\QQ_{\ell}}} \otimes_{\FF} \FF[\epsilon]/(\epsilon^2)$.
Thus, we get a contradiction to our assumption that $f$ is surjective.
Therefore, it follows that there exists no surjective morphism from $R \to \FF[\epsilon]/(\epsilon^2)$.
Recall that $R$ is an $\FF$-algebra.
So this implies that $R \simeq \FF$ which means $\mathfrak{m}_\ell = (p,z,\phi_1-\widehat{\bar\chi_1}(\text{Frob}_{\ell}))$.
This proves the lemma.
\end{proof}

Let $\mathcal{F}_{\ell} : W(\FF)\llbracket S, T \rrbracket \to R_{\rhob_c,\ell}$ be the morphism sending $S$ to $z$ and $T$ to $\phi_1-\widehat{\chibar_1}(\text{Frob}_{\ell})$.
We will now prove a result about the structure of $R_{\rhob_c,\ell}$ (similar to \cite[Lemma 3.10(ii)]{Bo2}).

\begin{lem}
\label{localstrlem}
The morphism $\mathcal{F}_{\ell}$ is surjective and $\ker(\mathcal{F}_{\ell}) = (SF_{\ell})$ for some non-zero $F_{\ell} \in (p,S,T) \subset W(\FF)\llbracket S, T \rrbracket$.
\end{lem}
\begin{proof}
Combining Lemma~\ref{maxlem} with \cite[Theorem 7.16(b)]{E}, we get that the morphism $\mathcal{F}_{\ell}$ is surjective.
From Lemma~\ref{tangentlem}, we know that $\dim(H^1(G_{\QQ_\ell},\ad(\rhob_c|_{G_{\QQ_{\ell}}})))=2$.
Since $\dim(H^0(G_{\QQ_\ell},\ad(\rhob_c|_{G_{\QQ_{\ell}}})))=1$, it follows, from local Euler characteristic formula, that $\dim(H^2(G_{\QQ_\ell},\ad(\rhob_c|_{G_{\QQ_{\ell}}})))=1$.
Hence, it follows, from \cite[Theorem 2.4]{Bo2}), that $\ker(\mathcal{F}_{\ell})$ is either principal or $(0)$.

Using the relation $g_{\ell}i_{\ell}g_{\ell}^{-1} = i_{\ell}^{\ell}$, we conclude that $(\phi-f_{\ell})z=0$ and $\phi^{-1}-f'_{\ell}=0$.
We now claim that $z \neq 0$. If $z=0$, then it follows, from the surjectivity of $\mathcal{F}_{\ell}$, that the dimension of the tangent space of $R_{\rhob_c,\ell}/(p)$ is at most $1$ which contradicts Lemma~\ref{tangentlem}.
This proves our claim. So we get that $\ker(\mathcal{F}_\ell) \neq (S)$.

 From the description of $P\rho_\ell(g_\ell)P^{-1}$ and $P\rho_\ell(i_\ell)P^{-1}$ given before Lemma~\ref{maxlem}, it follows that there exist characters $\xi_1, \xi_2 : G_{\QQ_{\ell}} \to (R_{\rhob_c,\ell}/(z))^{\times}$ lifting $\chibar_1|_{G_{\QQ_{\ell}}}$ and $\chibar_2|_{G_{\QQ_{\ell}}}$ such that $\tr(\rho_{\ell} \pmod{(z)}) = \xi_1 + \xi_2$.
As $p \nmid \ell-1$, it follows that $x \equiv 0 \pmod{(z)}$ and $w \equiv 0 \pmod{(z)}$.
Therefore, using induction, we obtain that $f_{\ell} \equiv \ell \pmod{(z)}$ and $f'_{\ell} \equiv \ell \pmod{(z)}$.

If $\ker(\mathcal{F}_{\ell})=0$, then $\mathcal{F}_{\ell}$ is an isomorphism and hence, $R_{\rhob_c,\ell}$ is an integral domain.
Since $z \neq 0$, this would imply that $\phi-f_{\ell} = 0$.
We will now prove $\phi-f_{\ell} \neq 0$ by contradiction.
Suppose $\phi-f_{\ell}=0$.
Since $\phi^{-1}-f'_{\ell}=0$, we get, from previous paragraph, that $\ell - \ell^{-1}=0$ in $R_{\rhob_c,\ell}/(z)$ which means $\ell^2-1=0$ in $R_{\rhob_c,\ell}/(z)$.

Let $V'$ be the free $W(\FF)$-module of rank $1$ on which $G_{\QQ_{\ell}}$ acts via the $p$-adic cyclotomic character $\chi_p$.
From the local Euler characteristic formula and \cite[Corollary 2.2]{T}, it follows that $\dim(H^1(G_{\QQ_{\ell}},V'))=1$ (as a $W(\FF)$-module) and the map $H^1(G_{\QQ_{\ell}},V') \to H^1(G_{\QQ_{\ell}}, \omega_p)$ induced by the natural surjective map $V' \to V'/pV'$ is surjective.
Therefore, there exists a representation $\tau : G_{\QQ_{\ell}} \to \GL_2(W(\FF))$ lifting $\rhob_c|_{G_{\QQ_\ell}}$ such that $\tau = \begin{pmatrix} \chi_p\widehat{\chibar_2} & * \\ 0 & \widehat{\chibar_1} \end{pmatrix}$ (see the proof of \cite[Proposition 3.4]{D} for more details).

Let $h : R_{\rhob_c,\ell} \to W(\FF)$ be the map induced by $\tau$.
Note that $\tr(\tau(i_{\ell})) =2$.
Hence, if $h(x)=a$, then $h(w)=-a$.
Moreover, we also have $\tr(\tau(g_\ell i_\ell)) =\tr(\tau(g_\ell))$.
So if $h(\phi_i) = \lambda_i$ for $i=1,2$, then we get $\lambda_2(1+a) + \lambda_1(1-a) = \lambda_2+\lambda_1$.
Thus, we get $a(\lambda_2 -\lambda_1)=0$.
Since $\phi_2 - \phi_1 \in R_{\rhob_c,\ell}^{\times}$, it follows that $a=0$.

Observe that $\det(\tau(i_{\ell})) = (1+a)(1-a) - h(yz) =1$. As $a=0$, we have $h(yz)=0$. Since $y \in R_{\rhob_c,\ell}^{\times}$, we get $h(z)=0$.
Therefore, $h : R_{\rhob_c,\ell} \to W(\FF)$ factors through $R_{\rhob_c,\ell}/(z)$. Since $\ell^2-1=0$ in $R_{\rhob_c,\ell}/(z)$, we get a contradiction.
Hence, we conclude that $\phi - f_{\ell} \neq 0$ and $\ker(\mathcal{F}_{\ell}) \neq (0)$.

Choose $G \in W(\FF) \llbracket S, T \rrbracket$ such that $\mathcal{F}_{\ell}(G) = \phi -f_{\ell}$.
Note that $SG \in \ker(\mathcal{F}_{\ell})$.
Now suppose $\ker(\mathcal{F}_{\ell}) \not\subset (S)$.
Since $W(\FF) \llbracket S, T \rrbracket$ is an UFD and $\ker(\mathcal{F}_{\ell})$ is a non-zero principal ideal, this implies that $G \in \ker(\mathcal{F}_{\ell})$.
This means $\phi - f_{\ell} = 0$ which gives us a contradiction.
Therefore, we conclude that $\ker(\mathcal{F}_{\ell}) \subset (S)$ which finishes the proof of the lemma.
\end{proof}

As before, $R(N\ell)_{\rhob_c}$ be the universal deformation ring of the representation $\rhob_c: G_{\QQ,N\ell p} \to \GL_2(\FF)$ in $\mathcal{C}$ and let $\rho_{N\ell}^{\univ} : G_{\QQ,N\ell p} \to \GL_2(R(N\ell)_{\rhob_c})$ be a (representation in the equivalence class giving the) universal deformation of $\rhob_c$. 
Let $n$ be the dimension of the tangent space of $R(N\ell)_{\rhob_c}/(p)$.
By \cite[Theorem 2.4]{Bo2}, we have a presentation of $R(N\ell)_{\rhob_c}$ of the following form:
\begin{equation}\label{presentation} 0 \to J \to W(\FF)\llbracket X_1,\cdots,X_n \rrbracket \xrightarrow{\mathcal{F'}} R(N\ell)_{\rhob_c} \to 0,\end{equation}
where $J$ is an ideal generated by at most $n-3$ elements.
By Lemma~\ref{dimlem}, $R(N\ell)_{\rhob_c}$ is a local complete intersection ring of Krull dimension $4$.
Hence, it follows that the minimum number of generators of $J$ is $n-3$.

Note that, $\rho_{N\ell}^{\univ}|_{G_{\QQ_{\ell}}}$ is a deformation of $\rhob_c|_{G_{\QQ_\ell}}$ and hence, it induces a morphism $\text{res}_{\ell} : R_{\rhob_c,\ell} \to R(N\ell)_{\rhob_c}$. Composing it with the morphism $\mathcal{F}_{\ell}$, gives us a morphism $W(\FF)\llbracket S, T \rrbracket \to R(N\ell)_{\rhob_c}$.
From \cite[Theorem 7.16(a)]{E}, it follows that this morphism lifts to a morphism $\Xi_{\ell} : W(\FF)\llbracket S, T \rrbracket \to W(\FF)\llbracket X_1,\cdots,X_n \rrbracket$ such that the following diagram commutes:
\begin{equation}
\label{diagram}
\begin{tikzcd}
W(\FF)\llbracket S, T \rrbracket  \arrow{r}{\mathcal{F}_{\ell}} \arrow[swap]{d}{\Xi_{\ell}} & R_{\rhob_c,\ell} \arrow{d}{\text{res}_{\ell}} \\
W(\FF)\llbracket X_1,\cdots,X_n \rrbracket \arrow{r}{\mathcal{F'}}& R(N\ell)_{\rhob_c}.
\end{tikzcd}
\end{equation}
Recall that $\ker(\mathcal{F}_{\ell}) = (SF_{\ell})$ for some non-zero, non-unit $F_{\ell} \in W(\FF)\llbracket S, T \rrbracket$ (see Lemma~\ref{localstrlem}).
Hence, $\Xi_{\ell}(SF_{\ell}) \in J$.
We will now prove a result which relates global obstructions to lifting $\rhob_c$ with the local obstructions (at $\ell$) to lifting $\rhob_c$.

\begin{prop}
\label{localstrprop}
Suppose $\chibar \neq \omega_p$. Let $\Xi_{\ell} : W(\FF)\llbracket S, T \rrbracket \to W(\FF)\llbracket X_1,\cdots,X_n\rrbracket$ be a morphism such that the diagram given in \eqref{diagram} commutes.
Then there exists $\overline{f_2},\cdots,\overline{f_{n-3}} \in W(\FF)\llbracket X_1,\cdots,X_n \rrbracket$ such that $\{\Xi_{\ell}(SF_{\ell}),\overline{f_2},\cdots,\overline{f_{n-3}}\}$ is a minimal set of generators of $J$.
\end{prop}
\begin{proof}
Since $n$ is the dimension of the tangent space of $R(N\ell)_{\rhob_c}/(p)$, it follows, from \cite[Theorem 2.4]{Bo2}, that $n = \dim(H^1(G_{\QQ,N\ell p},\ad(\rhob_c)))$ and $n-3= \dim(H^2(G_{\QQ,N\ell p},\ad(\rhob_c)))$.

We will now describe \cite[Theorem 3.1]{Bo3} which will be crucially used in the proof.
For every prime $q \mid Np$, let $R_{\rhob_c,q}$ be the \emph{versal} deformation ring of $\rhob_c|_{G_{\QQ_q}} : G_{\QQ_q} \to \GL_2(\FF)$ in $\mathcal{C}$.
Let $h^1_q = \dim(H^1(G_{\QQ_q},\ad(\rhob_c|_{G_{\QQ_q}})))$ and $h^2_q = \dim(H^2(G_{\QQ,N\ell p},\ad(\rhob_c|_{G_{\QQ_q}})))$.
By \cite[Theorem 2.4]{Bo2}, we have a presentation of $R_{\rhob_c,q}$ of the following form:
$$ 0 \to J_q \to W(\FF)\llbracket X_1,\cdots,X_{h^1_q} \rrbracket \xrightarrow{\mathcal{F}_q} R_{\rhob_c,q} \to 0,$$
where $J_q$ is an ideal generated by at most $h^2_q$ elements.

The deformation  $\rho_{N\ell}^{\univ}|_{G_{\QQ_{q}}}$ of $\rhob_c|_{G_{\QQ_q}}$ induces a morphism $\text{res}_{q} : R_{\rhob_c,q} \to R(N\ell)_{\rhob_c}$.
Let $\mathcal{F'} : W(\FF)\llbracket X_1,\cdots,X_n \rrbracket \to R(N\ell)_{\rhob_c}$ be the surjective map given in the exact sequence \eqref{presentation} above.
By the logic used for $\ell$ above, we get a morphism $\Xi_q : W(\FF)\llbracket X_1,\cdots,X_{h^1_q} \rrbracket \to W(\FF)\llbracket X_1,\cdots,X_n \rrbracket$ such that the following diagram commutes:
\[ \begin{tikzcd}
W(\FF)\llbracket X_1,\cdots,X_{h^1_q} \rrbracket  \arrow{r}{\mathcal{F}_{q}} \arrow[swap]{d}{\Xi_{q}} & R_{\rhob_c,q} \arrow{d}{\text{res}_{q}} \\
W(\FF)\llbracket X_1,\cdots,X_n \rrbracket \arrow{r}{\mathcal{F'}}& R(N\ell)_{\rhob_c}.
\end{tikzcd}
\]

For a prime $q \mid Np$, let $\{h_1,\cdots,h_{d_q}\}$ be a minimal set of generators of $J_q$. Note that $d_q \leq h^2_q = \dim(H^2(G_{\QQ_q}, \ad(\rhob_c|_{G_{\QQ_q}})))$.
Define $$\Sh^2(\ad(\rhob_c)) := \ker\big(H^2(G_{\QQ,N\ell p}, \ad(\rhob_c)) \to H^2(G_{\QQ_{\ell}},\ad(\rhob_c|_{G_{\QQ_\ell}})) \times \prod_{q \mid Np} H^2(G_{\QQ_{q}},\ad(\rhob_c|_{G_{\QQ_q}}))\big).$$
By \cite[Theorem 3.1]{Bo3}, there exists a set $\{g_1,\cdots,g_d\} \subset W(\FF)\llbracket X_1,\cdots,X_n \rrbracket$ with $d \leq \dim(\Sh^2(\ad(\rhob_c)))$ such that $J$ is generated by the set $$T_0 := \{\Xi_{\ell}(SF_{\ell})\} \bigcup \cup_{q \mid Np} \{\Xi_q(h_1),\cdots,\Xi_q(h_{d_q})\} \bigcup \{g_1,\cdots,g_d\}.$$
Recall that the minimum number of generators of $J$ is $n-3$.
Thus to prove the proposition, it suffices to prove that $|T_0| =n-3$.
As $\dim(H^2(G_{\QQ_\ell},\ad(\rhob_c|_{G_{\QQ_\ell}}))) =1$, this implies that
\begin{equation}
\label{jaast} 
n-3 \leq |T_0| = 1 + \sum_{q \mid Np} d_q + d \leq \sum_{q \mid N\ell p} \dim(H^2(G_{\QQ_q}, \ad(\rhob_c|_{G_{\QQ_q}}))) + \dim(\Sh^2(\ad(\rhob_c))).
\end{equation}

Note that $\ad(\rhob_c)^*$, the dual of $\ad(\rhob_c)$, is isomorphic to $\ad(\rhob_c)$. The semisimplification of $\ad(\rhob_c)$ is $\chibar \oplus \chibar^{-1} \oplus \mathbf{1}^{\oplus 2}$, where $\mathbf{1}$ is the trivial representation.
Recall that we have assumed $\chibar \neq \omega_p, \omega_p^{-1}$.
So we have $H^0(G_{\QQ, N\ell p}, \ad(\rhob_c)^* \otimes \omega_p) =0$.
Therefore, by Poitou-Tate exact sequence (\cite[8.6.10]{NSW}), we get that the map 
$$H^2(G_{\QQ,N\ell p}, \ad(\rhob_c)) \to H^2(G_{\QQ_{\ell}},\ad(\rhob_c|_{G_{\QQ_\ell}})) \times \prod_{q \mid Np} H^2(G_{\QQ_{q}},\ad(\rhob_c|_{G_{\QQ_q}}))$$
obtained by restriction onto each component is surjective.
Therefore, we get 
$$n-3 = \dim(H^2(G_{\QQ,N\ell p},\ad(\rhob_c))) = \sum_{q \mid N\ell p}\dim(H^2(G_{\QQ_q}, \ad(\rhob_c|_{G_{\QQ_q}}))) + \dim(\Sh^2(\ad(\rhob_c))).$$
Combining this with \eqref{jaast}, we conclude that $|T_0| \leq n-3$ and hence, $|T_0|=n-3$.
This proves the proposition.
\end{proof}

Let $\ell_1,\cdots,\ell_r$ be primes not dividing $N\ell p$ such that for every $1 \leq i \leq r$, $p \nmid \ell_i-1$ and $\chibar|_{G_{\QQ_{\ell_i}}} = \omega_p|_{G_{\QQ_{\ell_i}}}$.
Let $M:= N\ell\prod_{i=1}^{r}\ell_i$.
We will now prove a result which will combine analogues of Proposition~\ref{localstrprop} and Proposition~\ref{strprop} for $R(M)_{\rhob_c}$.

Let $R_0$ be the power series ring defined before Lemma~\ref{powerlem} for the tuple $(R(N\ell)_{\rhob_c}, M)$.
So it is the ring obtained by replacing $N$ by $N\ell$ in loc.cit.
Since $\dim(H^1_{\{p\}}(G_{\QQ, N\ell p},\chibar^{-1})) =1$, we can indeed make this replacement.
Let $\mathcal{F} :  R_0 \to R(M)_{\rhob_c}$ be the morphism constructed in Lemma~\ref{powerlem}.

Let $\text{res}_{\ell,M} : R_{\rhob_c,\ell} \to R(M)_{\rhob_c}$ be the morphism induced by the deformation $\rho^{\univ}_M|_{G_{\QQ_\ell}}$ of $\rhob_c|_{G_{\QQ_\ell}}$.
Repeating the argument given above for $R(N\ell)_{\rhob_c}$, we get a morphism $\Xi_{\ell,M} : W(\FF)\llbracket S, T \rrbracket \to R_0$ such that the following diagram commutes:
\begin{equation}
\label{diagram1}
 \begin{tikzcd}
W(\FF)\llbracket S, T \rrbracket  \arrow{r}{\mathcal{F}_{\ell}} \arrow[swap]{d}{\Xi_{\ell,M}} & R_{\rhob_c,\ell} \arrow{d}{\text{res}_{\ell,M}} \\
R_0 \arrow{r}{\mathcal{F}}& R(M)_{\rhob_c}.
\end{tikzcd}
\end{equation}

\begin{prop}
\label{localramprop}
In the description of $I_0:=\ker(\mathcal{F})$ obtained in Proposition~\ref{strprop}, $f_1$ can be taken to be $\Xi_{\ell,M}(SF_{\ell})$ in all the cases.
\end{prop}
\begin{proof}
Let $\Psi : R_0 \to W(\FF)\llbracket X_1,\cdots,X_n\rrbracket$ be the natural surjective map obtained by going modulo the ideal generated by the set:
\begin{enumerate}
\item $\{W_1,\cdots,W_r\}$, when $s=0$,
\item $\{U_1,\cdots,U_s\} \cup \{V_1,\cdots,V_s\} \cup \{W_{s+1},\cdots,W_r\}$, when $0 < s < r$.
\item  $\{U_1,\cdots,U_r\} \cup \{V_1,\cdots,V_r\}$, when $s=r$.
\end{enumerate}

Recall that $\Phi_{M,N\ell} : R(M)_{\rhob_c} \to R(N\ell)_{\rhob_c}$ is the surjective map induced by $\rho^{\univ}_{N\ell}$ when viewed as a representation of $G_{\QQ,Mp}$.

From Proposition~\ref{strprop}, it follows that there exists a surjective map $\mathcal{F''} : W(\FF)\llbracket X_1,\cdots,X_n\rrbracket \to  R(N\ell)_{\rhob_c}$ such that the following diagram commutes:
\begin{equation}
\label{diagram2}
 \begin{tikzcd}
R_0 \arrow{r}{\mathcal{F}} \arrow[swap]{d}{\Psi} & R(M)_{\rhob_c} \arrow{d}{\Phi_{M,N\ell}} \\
W(\FF)\llbracket X_1,\cdots,X_n\rrbracket \arrow{r}{\mathcal{F''}}& R(N\ell)_{\rhob_c}.
\end{tikzcd}
\end{equation}

Recall, from \S\ref{notsec}, that the maps $i_{\ell,M} : G_{\QQ_\ell} \to G_{\QQ,Mp}$ and $i_{\ell,N\ell} : G_{\QQ_\ell} \to G_{\QQ,N\ell p}$ are both induced by a fixed embedding $\iota_{\ell} : \overline{\QQ} \to \overline{\QQ_\ell}$.
So the following diagram commutes (where the top right arrow is the natural surjection $G_{\QQ,Mp} \to G_{\QQ,N\ell p}$):
\begin{equation}
\label{akruti}
\begin{tikzcd}
G_{\QQ,Mp}  \arrow{rr} && G_{\QQ,N\ell p}.\\
& G_{\QQ_\ell} \arrow[swap]{ur}{i_{\ell,N\ell}} \arrow{ul}{i_{\ell,M}} \\
\end{tikzcd}
\end{equation}
Hence by combining diagram~\eqref{akruti}, the universal property of $R_{\rhob_c,\ell}$ and definitions of $\text{res}_{\ell}$ and $\text{res}_{\ell,M}$, we get that the following diagram commutes:
\begin{equation}
\label{diagram3}
\begin{tikzcd}
R(M)_{\rhob_c}  \arrow{rr}{\Phi_{M,N\ell}} && R(N\ell)_{\rhob_c}.\\
& R_{\rhob_c,\ell} \arrow[swap]{ur}{\text{res}_{\ell}} \arrow{ul}{\text{res}_{\ell,M}} \\
\end{tikzcd}
\end{equation}

Therefore, by combining diagrams \eqref{diagram1}, \eqref{diagram2} and \eqref{diagram3}, we conclude that $\Psi \circ \Xi_{\ell,M}$ is a lift of $\text{res}_{\ell} \circ \mathcal{F}_\ell$ to $W(\FF)\llbracket X_1,\cdots,X_n\rrbracket $ i.e. the following diagram commutes:
\[ \begin{tikzcd}
W(\FF)\llbracket S, T \rrbracket  \arrow{r}{\mathcal{F}_{\ell}} \arrow[swap]{d}{\Psi \circ \Xi_{\ell,M}} & R_{\rhob_c,\ell} \arrow{d}{\text{res}_{\ell}} \\
W(\FF)\llbracket X_1,\cdots,X_n \rrbracket \arrow{r}{\mathcal{F''}}& R(N\ell)_{\rhob_c}.
\end{tikzcd}
\]

From Proposition~\ref{localstrprop}, we know that there exists $\overline{f_2},\cdots,\overline{f_{n-3}} \in W(\FF) \llbracket X_1,\cdots,X_n\rrbracket$ such that $\ker(\mathcal{F''}) = (\Psi \circ \Xi_{\ell,M}(SF_\ell),\overline{f_2},\cdots,\overline{f_{n-3}})$ and the set $\{\Psi \circ \Xi_{\ell,M}(SF_\ell),\overline{f_2},\cdots,\overline{f_{n-3}}\}$ is a minimal set of generators of $\ker(\mathcal{F''})$.
Note that Proposition~\ref{localstrprop} holds for any surjective morphism $W(\FF)\llbracket X_1,\cdots,X_n\rrbracket \to  R(N\ell)_{\rhob_c}$ as the conclusion of \cite[Theorem 3.1]{Bo3} does not depend on the choice of this surjective morphism.
So we can apply it here.

Therefore, using Proposition~\ref{strprop} and Nakayama's lemma, we conclude that the elements $f_1,\cdots,f_{n-3}$ appearing in the description of $I_0$ in Proposition~\ref{strprop} can be chosen such that $\Psi(f_1) = \Psi \circ \Xi_{\ell,M}(SF_\ell)$ and $\Psi(f_j) =\overline{f_j}$ for all $2 \leq j \leq n-3$.

Thus there exists a $\Theta \in \ker(\Psi)$ such that $f_1=\Xi_{\ell,M}(SF_\ell) + \Theta$.
Let $S_0$ be the set of generators of $I_0$ appearing in Proposition~\ref{localstrprop} with $f_1,\cdots,f_{n-3}$ chosen as in the previous paragraph.
Let $I'_0$ be the ideal of $R_0$ generated by the set $S_0 \setminus\{f_1\}$.
By Lemma~\ref{localstrlem}, we know that $\Xi_{\ell,M}(SF_{\ell}) \in \ker(\mathcal{F})$.
Therefore, there exists a $\omega \in R_0$ and $e \in I'_0$ such that \begin{equation}\label{barobar}\Xi_{\ell,M}(SF_{\ell}) = \omega(\Xi_{\ell,M}(SF_\ell) + \Theta) + e.\end{equation}

Suppose $\omega \not\in R_0^{\times}$. Then $1 - \omega \in R_0^{\times}$ and hence, it follows, from \eqref{barobar}, that $\Xi_{\ell,M}(SF_{\ell})$ lies in the ideal generated by $S'_0 := (S_0 \setminus\{f_1\}) \cup \{\Theta\}$.
Note that, from Proposition~\ref{strprop} and the description of $S_0$ obtained there, it follows that $\ker(\mathcal{F''})$ is generated by the set $\Psi(\ker(\mathcal{F}))$.
Hence, $\ker(\mathcal{F''})$ is generated by the set $\Psi(S'_0)$. 
Since $\Theta \in \ker(\Psi)$, the description of $S_0$ obtained in Proposition~\ref{strprop} implies that $\ker(\mathcal{F''})$ is generated by $\{\overline{f_2},\cdots,\overline{f_{n-3}}\}$.
This contradicts the fact that the minimum number of generators of $\ker(\mathcal{F''})$ is $n-3$ (see the discussion after Lemma~\ref{localstrlem}).

Therefore, we conclude that $\omega \in R_0^{\times}$. This means that $\ker(\mathcal{F})$ is generated by the set $(S_0 \setminus\{f_1\}) \cup \{\Xi_{\ell,M}(SF_\ell)\}$ which proves the proposition.
\end{proof}

 \section{Proof of Theorem~\ref{thmb}}
\label{proofthm}
We are now ready to prove Theorem~\ref{thmb}. We keep the notation established in the previous section. Let $M=N\prod_{i=1}^{r}\ell_i$ and $s$ be the number of primes in the set $\{\ell_1,\cdots,\ell_r\}$ which are $-1 \pmod{p}$. If $s \neq 0$, then we assume, without loss of generality, that $p \mid \ell_j+1$ for all $1 \leq j \leq s$. 
Let $k>2$ be an integer such that $k \equiv k_0 \pmod{p-1}$ and $\alpha,\beta,\delta_k$ be the elements of $R(M)_{\rhob_c}$ found in Lemma~\ref{diagonallem}.
If $s \neq 0$, then let $u_1,\cdots,u_s$ be the elements of $R(M)_{\rhob_c}$ found in Lemma~\ref{ramlem}.

We now define an ideal $\mathcal{I}_0$ of $R(M)_{\rhob_c}$ in the following way:
\begin{enumerate}
\item If $s=0$, then define $\mathcal{I}_0 = (\alpha,\beta,\delta_k, \psi_{1,1}-\ell_1\psi_{2,1},\cdots,\psi_{1,r}-\ell_r\psi_{2,r})$ ,
\item If $0<s<r$, then define $$\mathcal{I}_0 = (\alpha,\beta,\delta_k,u_1,\cdots,u_s,h_{\ell_1}-\psi_1^{-1},\cdots,h_{\ell_s}-\psi_s^{-1},\psi_{1,s+1}-\ell_{s+1}\psi_{2,s+1},\cdots,\psi_{1,r}-\ell_r\psi_{2,r}),$$
\item If $s=r$, then define $\mathcal{I}_0 = (\alpha,\beta,\delta_k,u_1,\cdots,u_r,h_{\ell_1}-\psi_1^{-1},\cdots,h_{\ell_r}-\psi_r^{-1})$.
\end{enumerate}

Let $R:= R(M)_{\rhob_c}/\mathcal{I}_0$.
From Lemma~\ref{surjlemma}, we get that $R$ is a finite $W(\FF)$-algebra and hence, its Krull dimension is at most $1$.
Let $\mathcal{F} : R_0 \to R(M)_{\rhob_c}$ be the surjective map constructed in Lemma~\ref{powerlem}.
Choose $f_{\alpha}, f_{\beta}, f_{\delta_k} \in R_0$ such that $\mathcal{F}(f_{\alpha})= \alpha$, $\mathcal{F}(f_{\beta})= \beta$ and $\mathcal{F}(f_{\delta_k})= \delta_k$.
Let $\mathcal{H'} : R_0 \to R$ be the map obtained by composing $\mathcal{F}$ with the quotient map $R(M)_{\rhob_c} \to R$.

From the description of $R_0$ given before Lemma~\ref{powerlem} and the description of $\ker(\mathcal{F})$ obtained in Proposition~\ref{strprop}, we get that:
\begin{enumerate}
\item If $s=0$, then $R_0$ has Krull dimension $n+r+1$ and $\ker(\mathcal{H'})$ is generated by the set $$\{f_{\alpha},f_{\beta},f_{\delta_k},f_1,\cdots,f_{n-3},g_1,\cdots,g_r\}.$$
So it is generated by $n+r$ elements.
\item If $0 < s < r$, then $R_0$ has Krull dimension $n+s+r+1$ and $\ker(\mathcal{H'})$ is generated by the set 
$$\{f_{\alpha},f_{\beta},f_{\delta_k},f_1,\cdots,f_{n-3},U_1,\cdots,U_s,h'_1,\cdots,h'_s,g_{s+1},\cdots,g_r\}.$$
So it is generated by $n+s+r$ elements.
\item If $s=r$, then $R_0$ has Krull dimension $n+2r+1$ and $\ker(\mathcal{H'})$ is generated
the set $$\{f_{\alpha},f_{\beta},f_{\delta_k},f_1,\cdots,f_{n-3},U_1,\cdots,U_r,h'_1,\cdots,h'_r\}.$$
 So it is generated by $n+2r$ elements.
\end{enumerate}
Therefore, in each case, using \cite[Theorem 10.2]{E}, we obtain that the Krull dimension of $R$ is at least $1$.
Hence, we conclude that $R$ is a finite $W(\FF)$-algebra of Krull dimension $1$.

Let $P$ be a minimal prime of $R$. So $R/P$ is an integral domain which is finite over $W(\FF)$. Hence, it is isomorphic to a subring of $\overline{\QQ_p}$. So fix an inclusion $i : R/P \to \overline{\QQ_p}$.
Let $\rho_P : G_{\QQ,Mp} \to \GL_2(R/P)$ be the representation obtained by composing $\rho_M^{\univ}$ with the natural surjective map $R(M)_{\rhob_c} \to R/P$. 
Using the inclusion $i$, we can view $\rho_P$ as a representation over $\overline{\QQ_p}$.

\begin{prop}
\label{modularprop}
Let $P$ be a minimal prime of $R$ and $\rho_P$ be the corresponding representation of $G_{\QQ,Mp}$ as above.
Then $\rho_P$ is the $p$-adic Galois representation attached to a newform $f$ of level $M'$ such that $\ell_i \mid\mid M'$ and $p \nmid M'$.
\end{prop}
\begin{proof}
We will prove the existence of $f$ using the modularity lifting theorem of Skinner--Wiles (\cite[Theorem A]{SW2}).
In order to apply their theorem, first observe that, by Lemma~\ref{diagonallem}, $\rho_P : G_{\QQ,Mp} \to \GL_2(\overline{\QQ_p})$ is ordinary at $p$ i.e. $\rho_P|_{G_{\QQ_p}} = \begin{pmatrix} \eta_2 & 0\\ *& \eta_1\end{pmatrix}$, where $\eta_2$ an unramified character of $G_{\QQ_p}$. Moreover, $\det(\rho_P)=\epsilon_k\chi_p^{k-1}$, where $\epsilon_k$ is a character of $G_{\QQ,Mp}$ of finite order. 
We claim that $\rho_P$ is irreducible. 

If $\rho_P$ is reducible, then, by combining Brauer-Nesbitt Theorem and the proof of Lemma~\ref{redlem}, we get that there exist characters $\chi_1,\chi_2 : G_{\QQ,Mp} \to (R/P)^{\times}$ such that $\chi_i$ is a lift of $\chibar_i$ for $i=1,2$ and $\rho_P = \begin{pmatrix} \chi_2 & * \\ 0  & \chi_1 \end{pmatrix}$ (see proof of Proposition~\ref{oneprimelem} for more details).
Therefore, we have $\chi_2|_{G_{\QQ_p}} = \eta_2$. As $\eta_2$ is unramified at $p$ and $p \nmid \phi(M)$, it follows that $\chi_2 = \widehat\chibar_2$ and $\chi_1$ is unramified at every $\ell_i$.

If $p \nmid \ell_i+1$, then the image of $\psi_{1,i}-\ell_i\psi_{2,i}$ in $R$ is $0$. 
If $p \mid \ell_i+1$, then the images of $u_i$ and $h_{\ell_i}-\psi_i^{-1}$ in $R$ are $0$.
As $h_{\ell_i} \equiv \ell_i \pmod{u_iv_i}$, it follows that the image of $\psi_i^{-1} - \ell_i$ in $R$ is $0$.
Therefore, we get that $\chi_1\chi_2^{-1}|_{G_{\QQ_{\ell_i}}} = \chi_p|_{G_{\QQ_{\ell_i}}}$.
It follows, from the fact that $\chi_2$ is a character of finite order, that $\det(\rho_P)\chi_p^{-1}|_{G_{\QQ_{\ell_i}}}$ is a character of finite order.
Since $k >2$, the description of $\det(\rho_P)$ obtained in the previous paragraph gives a contradiction.

Therefore, we conclude that $\rho_P$ is irreducible.
Hence, by \cite[Theorem A]{SW2}, $\rho_P$ is the $p$-adic Galois representation attached to a newform $f$ of some level $M'$.
Since $\rho_P$ is ordinary, it follows, from \cite[Proposition 3.6]{Mo}, that $f$ is $p$-ordinary i.e. $U_p$-eigenvalue of $f$ is a $p$-adic unit. 
As $k > 2$, \cite[Lemma 5.1.2]{D2} implies that $p \nmid M'$.

From the analysis given above, it follows that the semi-simplification of $\rho_f |_{G_{\QQ_{\ell_i}}}$ is $\chi_{\ell_i} \oplus \chi_{\ell_i}\chi_p$ for some character $\chi_{\ell_i}$.
As $p \nmid \ell_i-1$ and both $\chibar_1$ and $\chibar_2$ are unramified at $\ell_i$, it follows that $\chi_{\ell_i}$ is an unramified character of $G_{\QQ_{\ell_i}}$.
If $\ell_i \nmid M'$, then $\rho_f$ is unramified at $\ell_i$.
This implies that $\rho_f |_{G_{\QQ_{\ell_i}}} = \chi_{\ell_i} \oplus \chi_{\ell_i}\chi_p$.
But we know, from \cite[Lemma 5.1.1]{D2}, that this is not possible. Hence, it follows that $\ell_i \mid M'$ for all $1 \leq i \leq r$.
Combining this with the description of the semi-simplification of $\rho_f |_{G_{\QQ_{\ell_i}}}$ obtained above, we conclude that $\rho_f|_{G_{\QQ_{\ell_i}}}$ is Steinberg for all $1 \leq i \leq r$ i.e. $\rho_f|_{G_{\QQ_{\ell_i}}} \simeq \begin{pmatrix} \chi_i\chi_p & * \\ 0 & \chi_i\end{pmatrix}$, where $\chi_i$ is an unramified character and $*$ is ramified.
Hence, $\ell_i \mid\mid M'$ for all $1 \leq i \leq r$.
\end{proof}

Now $\rho_f$ is unramified outside the set of primes dividing $Mp$ and $p \nmid M'$.
Let $q$ be a prime dividing $\frac{M'}{\prod_{i=1}^{r}\ell_i}$.
 Then $q \mid N$.
Moreover, as $p \nmid \phi(N)$, Hypotheses~\eqref{hyp2} and \eqref{hyp3} of Set-up~\ref{setup} and \cite[Proposition 2]{Ca} together imply that if $q^e$ is the highest power of $q$ dividing $\frac{M'}{\prod_{i=1}^{r}\ell_i}$, then $q^e \mid N$.
Therefore, we get that $\prod_{i=1}^{r}\ell_i \mid M' \mid M=N\prod_{i=1}^{r}\ell_i$.
This completes the proof of the theorem.

\section{Proof of Theorem~\ref{thmc}}
\label{proofthm1}
Now we move on to Theorem~\ref{thmc}. 
We keep the notation from the previous section.
Since $H^1_{\{p\}}(G_{\QQ,Np},\chibar^{-1})=0$ and $\chibar^{-1}|_{G_{\QQ_{\ell_0}}} = \omega_p|_{G_{\QQ_{\ell_0}}}$, it follows from Greenberg--Wiles formula (\cite[Theorem 2]{Wa}) that $$\dim(H^1_{\{p\}}(G_{\QQ,N\ell_0 p},\chibar^{-1})) = \dim(H^1_{\{p\}}(G_{\QQ,Np},\chibar^{-1})) + \dim(H^0(G_{\QQ_{\ell_0}},\chibar\omega_p|_{G_{\QQ_{\ell_0}}})) = 0+1=1.$$
Let $c \in H^1_{\{p\}}(G_{\QQ,N\ell_0 p},\chibar^{-1})$ be a non-zero element. 

First suppose $p \nmid \ell_0 +1$.
Note that $p \nmid \phi(N\ell_0)$ and $N\ell_0$ satisfies the conditions of Setup~\ref{setup}.
Therefore by Theorem~\ref{thmb}, it follows that there exists a newform $f$ of level $M' \mid N\prod_{i=0}^{r}\ell_i$ and weight $k$ such that $\rho_f$ lifts $\rhob_0$ and $f$ is new at $\ell_i$ for all $1 \leq i \leq r$.
Moreover, from the proof of Theorem~\ref{thmb}, we see that there exists a $G_{\QQ,N\prod_{i=0}^{r}\ell_i p}$-stable lattice of $\rho_f$ such that the corresponding representation $\rho$ of $G_{\QQ,N\prod_{i=0}^{r}\ell_i p}$ is a deformation of $\rhob_c$.
As $H^1_{\{p\}}(G_{\QQ,Np},\chibar^{-1}) = 0$, it follows that $c$ is ramified at $\ell_0$.
Therefore, it follows that $\rho$ is ramified at $\ell_0$ and hence, $\rho_f$ is ramified at $\ell_0$ which means $f$ is new at $\ell_0$.
This proves the theorem when $p \nmid \ell_0 +1$.

Now suppose $p \mid \ell_0+1$. We will follow the strategy of the proof of Theorem~\ref{thmb} here with a slight modification.
Let $M=N\prod_{i=0}^{r}\ell_i$ and $\mathcal{I}_0$ be the ideal of $R(M)_{\rhob_c}$ defined in the proof of Theorem~\ref{thmb} (\S\ref{proofthm}).
For the ease of notation, we will refer to $\ell_0$ by $\ell$ for the rest of this section and we will use the notation introduced in diagram~\eqref{diagram1}.

Let $\mathcal{J}_0$ be the ideal of $R(M)_{\rhob_c}$ generated by $\mathcal{I}_0$ and $\text{res}_{\ell,M} \circ \mathcal{F}_{\ell}(S)$.
Let $R' := R(M)_{\rhob_c}/\mathcal{J}_0$.
From Lemma~\ref{surjlemma}, we get that $R'$ is a finite $W(\FF)$-algebra and hence, its Krull dimension is at most $1$.

Note that, by Proposition~\ref{localramprop}, we know that there exist elements $f_2,\cdots,f_{n-3}$ of $R_0$ such that $\ker(\mathcal{F})$ is generated by $\Xi_{\ell,M}(SF_\ell),f_2,\cdots,f_{n-3}$ along with the other elements described in Proposition~\ref{strprop}.
For instance, when $s=0$, $\ker(\mathcal{F})$ is generated by the set $\{\Xi_{\ell,M}(SF_\ell),f_2,\cdots,f_{n-3},W_1g_1,\cdots,W_rg_r\}$.
One gets similar statements for other cases.

Recall that $ \text{res}_{\ell,M} \circ \mathcal{F}_{\ell} =\mathcal{F} \circ \Xi_{\ell,M}$ (see diagram~\eqref{diagram1}).
Let $f_{\alpha}, f_{\beta}$ and $f_{\delta_k}$ be the elements of $R_0$ as defined in \S\ref{proofthm}.
As in the previous section, let $\mathcal{H''} : R_0 \to R'$ be the map obtained by composing $\mathcal{F}$ with the quotient map $R(M)_{\rhob_c} \to R'$.

Combining the description of $R_0$ given before Lemma~\ref{powerlem}, the description of $\ker(\mathcal{F})$ obtained in Proposition~\ref{strprop} and Proposition~\ref{localramprop}, we get that:
\begin{enumerate}
\item If $s=0$, then $R_0$ has Krull dimension $n+r+1$ and $\ker(\mathcal{H''})$ is generated by the set $$\{f_{\alpha},f_{\beta},f_{\delta_k},\Xi_{\ell,M}(S),f_2,\cdots,f_{n-3},g_1,\cdots,g_r\}.$$
So it is generated by $n+r$ elements.
\item If $0 < s < r$, then $R_0$ has Krull dimension $n+s+r+1$ and $\ker(\mathcal{H''})$ is generated by the set 
$$\{f_{\alpha},f_{\beta},f_{\delta_k},\Xi_{\ell,M}(S),f_2,\cdots,f_{n-3},U_1,\cdots,U_s,h'_1,\cdots,h'_s,g_{s+1},\cdots,g_r\}.$$
So it is generated by $n+s+r$ elements.
\item If $s=r$, then $R_0$ has Krull dimension $n+2r+1$ and $\ker(\mathcal{H''})$ is generated
the set $$\{f_{\alpha},f_{\beta},f_{\delta_k},\Xi_{\ell,M}(S),f_2,\cdots,f_{n-3},U_1,\cdots,U_r,h'_1,\cdots,h'_r\}.$$
 So it is generated by $n+2r$ elements.
\end{enumerate}
Therefore, in each case, using \cite[Theorem 10.2]{E}, we conclude that $R'$ is a finite $W(\FF)$-algebra of Krull dimension $1$.

Let $P$ be a minimal prime of $R'$. Thus $R'/P$ is a subring of $\overline{\QQ_p}$ (see \S\ref{proofthm} for more details). Fix an inclusion $i : R'/P \to \overline{\QQ_p}$.
Let $\rho_P : G_{\QQ,Mp} \to \GL_2(R'/P)$ be the representation obtained by composing $\rho_M^{\univ}$ with the natural surjective map $R(M)_{\rhob_c} \to R'/P$. 
Using the inclusion $i$, we can view $\rho_P$ as a representation over $\overline{\QQ_p}$.

As both $R:= R(M)_{\rhob_c}/\mathcal{I}_0$ and $R'$ are finite $W(\FF)$-algebras of Krull dimension $1$, it follows that there exists a minimal prime ideal $Q$ of $R$ such that the quotient map $R \to R'$ induces an isomorphism $R/Q \simeq R'/P$.
Therefore, using Proposition~\ref{modularprop}, we conclude that $\rho_P$ is the $p$-adic Galois representation attached to a $p$-ordinary newform $f$ of some level $M'$ such that $\ell_i \mid\mid M'$ for every $1 \leq i \leq r$ and $p \nmid M'$.
Note that, in this case, we are applying Proposition~\ref{modularprop} after replacing $N$ by $N\ell$ in loc. cit. and taking $M= N\ell\prod_{i=1}^{r}\ell_i$.
Therefore, we do not get any statement about the divisibility of $M'$ by $\ell$ from it.

Note that $\rho_P : G_{\QQ,Mp} \to \GL_2(R'/P)$ is a deformation of $\rhob_c$ which is ramified at $\ell$. 
Therefore, $\rho_P$ is ramified at $\ell$ which implies that $\ell \mid M'$.
Since the image of $\Xi_{\ell,M}(S)$ (i.e. the image of $\text{res}_{\ell,M} \circ \mathcal{F}_{\ell}(S)$) in $R'/P$ is $0$, it follows, from the proof of Lemma~\ref{localstrlem}, that $\rho_P|_{G_{\QQ_\ell}} \simeq \begin{pmatrix} \omega_1 & * \\0 & \omega_2 \end{pmatrix}$, where for $i=1,2$, $\omega_i : G_{\QQ_\ell} \to (R'/P)^{\times}$ is an unramified character lifting $\chibar_i|_{G_{\QQ_\ell}}$.
Therefore, $\rho_P|_{G_{\QQ_\ell}}$ is either principal series or Steinberg.

Suppose $\rho_P|_{G_{\QQ_\ell}}$ is principal series.
Then, it is semi-simple and hence, it follows, from Brauer-Nesbitt theorem, that $\rho_P|_{G_{\QQ_\ell}} =\omega_1 \oplus \omega_2$.
Thus, it implies that $\rho$ is unramified at $\ell$ which gives us a contradiction.
Therefore, $\rho_P|_{G_{\QQ_{\ell}}}$ is Steinberg which implies that $\ell \mid\mid M'$. 
Note that $p \nmid \phi(N\ell^2)$ and $N\ell^2$ satisfies Hypotheses~\eqref{hyp2} and \eqref{hyp3} of Set-up~\ref{setup} .
Hence, it follows, from the proof of Theorem of Theorem~\ref{thmb}, that $M' \mid N\ell^2\prod_{i=1}^{r}\ell_i$.
As we have proved $\ell \mid\mid M'$, it follows that $M' \mid N\ell\prod_{i=1}^{r}\ell_i$.
This finishes the proof of the theorem.

\section{Proof of Theorem~\ref{thme}}
\label{proofthm2}
Since $\dim(H^1_{\{p\}}(G_{\QQ,N\ell p},\chibar^{-1}))=1$ and $p \nmid \ell-1$, Lemma~\ref{dimlem} implies that $R(N\ell)_{\rhob_c}$ is a local complete intersection ring of Krull dimension $4$.
Let $\overline{f_2},\cdots,\overline{f_{n-3}} \in W(\FF)\llbracket X_1,\cdots,X_n\rrbracket$ be the elements found in Proposition~\ref{localstrprop}.
Combining Proposition~\ref{localstrprop} and the fact that $\mathcal{F'} \circ \Xi_\ell = \text{res}_{\ell} \circ \mathcal{F}_\ell$ (see diagram~\eqref{diagram}), it follows that the kernel of the surjective map $W(\FF)\llbracket X_1,\cdots,X_n\rrbracket \to R(N\ell)_{\rhob_c}/(\text{res}_{\ell} \circ \mathcal{F}_{\ell}(F_\ell))$, obtained by composing $\mathcal{F'}$ with the quotient map $R(N\ell)_{\rhob_c} \to R(N\ell)_{\rhob_c}/(\text{res}_{\ell} \circ \mathcal{F}_{\ell}(F_\ell))$, is generated by the set $\{\Xi_{\ell}(F_\ell),\overline{f_2},\cdots,\overline{f_{n-3}}\}$.

Since $R(N\ell)_{\rhob_c}$ has Krull dimension $4$, we conclude, using \cite[Theorem 10.2]{E}, that the Krull dimension of $R(N\ell)_{\rhob_c}/(\text{res}_{\ell} \circ \mathcal{F}_{\ell}(F_\ell))$ is also $4$.
Let $k >2$ be an integer such that $k \equiv k_0 \pmod{p-1}$ and let $\alpha$, $\beta$ and $\delta_k$ be elements of $R(N\ell)_{\rhob_c}$ found in Lemma~\ref{diagonallem}.
Hence, it follows, by combining \cite[Theorem 10.2]{E} and Lemma~\ref{dimlem}, that $\mathcal{S} := R(N\ell)_{\rhob_c}/(\alpha,\beta,\delta_k,\text{res}_{\ell} \circ \mathcal{F}_{\ell}(F_\ell))$ is a finite $W(\FF)$-algebra of Krull dimension $1$.

Let $Q$ be a minimal prime of $\mathcal{S}$ and $\rho_Q : G_{\QQ,N\ell p} \to \GL_2(\mathcal{S}/Q)$ be the representation obtained by composing $\rho_{N\ell}^{\univ}$ with the natural surjective map $R(N\ell)_{\rhob_c} \to \mathcal{S}/Q$.
Since $\rhob_c$ is ramified at $\ell$, $\rho_Q$ is also ramified at $\ell$.
As $\mathcal{S}/Q$ is a finite $W(\FF)$-algebra and an integral domain of Krull dimension $1$, it is isomorphic to a subring of $\overline{\QQ_p}$.
Fix an embedding $\mathcal{S}/Q \to \overline{\QQ_p}$ and using this embedding, we view $\rho_Q$ as a representation over $\overline{\QQ_p}$. 

By Lemma~\ref{diagonallem}, $\rho_Q : G_{\QQ,N\ell p} \to \GL_2(\overline{\QQ_p})$ is ordinary at $p$ i.e. $\rho_Q|_{G_{\QQ_p}} = \begin{pmatrix} \eta_2 & 0\\ * & \eta_1\end{pmatrix}$, where $\eta_2$ an unramified character of $G_{\QQ_p}$. Moreover, $\det(\rho_Q)=\epsilon_k\chi_p^{k-1}$, where $\epsilon_k$ is a character of $G_{\QQ,Mp}$ of finite order. 
Since $k > 2$ and $p \nmid \phi(N\ell)$, we get, by following the proof of Proposition~\ref{modularprop}, that $\rho_Q$ is irreducible.
Hence, by \cite[Theorem A]{SW2}, $\rho_Q$ is the $p$-adic Galois representation attached to a newform $f$ of some level $M$.
So, it follows, from \cite[Proposition 3.6]{Mo}, that $f$ is $p$-ordinary. As $k > 2$, \cite[Lemma 5.1.2]{D2} implies that $p \nmid M$.

Since $\rho_Q$ is ramified at $\ell$, it follows that $\ell \mid M$. As $\chibar_1$ and $\chibar_2$ are unramified at $\ell$, \cite[Proposition 2]{Ca} implies that $\ell^3 \nmid M$. 
Now $\rho_Q$ is unramified outside the set of primes dividing $N\ell p$ and $p \nmid M$.
Let $q \neq \ell$ be a prime dividing $M$.
 Then $q \mid N$.
Moreover, as $p \nmid \phi(N)$, Hypotheses~\eqref{hyp2} and \eqref{hyp3} of Set-up~\ref{setup} and \cite[Proposition 2]{Ca} together imply that if $q^e$ is the highest power of $q$ dividing $M$, then $q^e \mid N$.
So it follows that $M \mid N\ell^2$.
Hence, to prove the theorem, it suffices to prove that $\ell^2 \mid M$.

Now if $\ell \mid\mid M$, then, using \cite[Proposition 2]{Ca} again, we get that $\rho_Q|_{G_{\QQ_{\ell}}} \simeq \begin{pmatrix} \chi\chi_p & * \\ 0 & \chi\end{pmatrix}$, where $\chi$ is an unramified character and $*$ is ramified.
For an element $a \in R(N\ell)_{\rhob_c}$, denote its image in $\mathcal{S}/Q$ by $\bar a$.
We will now use the notation developed just before Lemma~\ref{localstrlem}.
So, from the proof of Lemma~\ref{localstrlem}, it follows that $\overline{\text{res}_{\ell}(z)}=0$.
In the same proof, we obtain that $f_{\ell} \equiv f'_\ell \equiv \ell \pmod{(z)}$ and $\phi^{-1}-f'_\ell=0$.
Therefore, we conclude that $\overline{\text{res}_{\ell}(\phi)}^{-1}-\bar\ell =0$.

Recall that $z(\phi-f_{\ell})=0$.
Let $H_\ell \in W(\FF)\llbracket S, T \rrbracket$ be an element such that $\mathcal{F}_\ell(H_\ell)=\phi-f_\ell$.
So $SH_\ell \in \ker(\mathcal{F}_\ell)$.
It follows, from Lemma~\ref{localstrlem}, that $SF_\ell \mid SH_\ell$. This means $F_\ell \mid H_\ell$ as $W(\FF)\llbracket S, T \rrbracket$ is a UFD.
Hence, we conclude that $\mathcal{F}_\ell(H_\ell) = \phi - f_\ell \in (\mathcal{F}_{\ell}(F_\ell))$.
Therefore, we obtain that $\overline{\text{res}_{\ell}(\phi)}-\bar\ell =0$.

Thus we get, $\bar\ell=\bar\ell^{-1}$ i.e. $\bar\ell+1=0$.
As $\mathcal{S}/Q$ is a finite $W(\FF)$-algebra, we get that $\mathcal{S}/Q$ has Krull dimension $0$.
However, we know that $\mathcal{S}/Q$ has Krull dimension $1$ which gives a contradiction.
Hence, we conclude that $\ell^2 \mid M$ which proves part~\eqref{bhaag1} of the theorem.

To prove part~\eqref{bhaag2} of the theorem, we follow the strategy used in the proof of Theorem~\ref{thmc} with a slight modification.
Let $M = N\ell\prod_{i=1}^{r}\ell_i$ and $\mathcal{I}_0$ be the ideal of $R(M)_{\rhob_c}$ defined in \S\ref{proofthm}.
Let $\mathcal{J}'_0$ be the ideal of $R(M)_{\rhob_c}$ generated by $\mathcal{I}_0$ and $\text{res}_{\ell,M} \circ \mathcal{F}_{\ell}(F_\ell)$.
Let $R'' := R(M)_{\rhob_c}/\mathcal{J}'_0$.
From Lemma~\ref{surjlemma}, we get that $R''$ is a finite $W(\FF)$-algebra and hence, its Krull dimension is at most $1$.

Note that, by Proposition~\ref{localramprop}, we know that there exist elements $f_2,\cdots,f_{n-3}$ of $R_0$ such that $\ker(\mathcal{F})$ is generated by $\Xi_{\ell,M}(SF_\ell),f_2,\cdots,f_{n-3}$ along with the other elements described in Proposition~\ref{strprop}.
Let $f_{\alpha}, f_{\beta}$ and $f_{\delta_k}$ be the elements of $R_0$ as defined in \S\ref{proofthm} and let $\mathcal{H'''} : R_0 \to R''$ be the map obtained by composing $\mathcal{F}$ with the quotient map $R(M)_{\rhob_c} \to R''$.

Combining the description of $R_0$ given before Lemma~\ref{powerlem}, the description of $\ker(\mathcal{F})$ obtained in Proposition~\ref{strprop} and Proposition~\ref{localramprop}, we get that (see proof of Theorem~\ref{thmc} for more details):
\begin{enumerate}
\item If $s=0$, then $R_0$ has Krull dimension $n+r+1$ and $\ker(\mathcal{H'''})$ is generated by the set $$\{f_{\alpha},f_{\beta},f_{\delta_k},\Xi_{\ell,M}(F_\ell),f_2,\cdots,f_{n-3},g_1,\cdots,g_r\}.$$
So it is generated by $n+r$ elements.
\item If $0 < s < r$, then $R_0$ has Krull dimension $n+s+r+1$ and $\ker(\mathcal{H'''})$ is generated by the set 
$$\{f_{\alpha},f_{\beta},f_{\delta_k},\Xi_{\ell,M}(F_\ell),f_2,\cdots,f_{n-3},U_1,\cdots,U_s,h'_1,\cdots,h'_s,g_{s+1},\cdots,g_r\}.$$
So it is generated by $n+s+r$ elements.
\item If $s=r$, then $R_0$ has Krull dimension $n+2r+1$ and $\ker(\mathcal{H'''})$ is generated
the set $$\{f_{\alpha},f_{\beta},f_{\delta_k},\Xi_{\ell,M}(F_\ell),f_2,\cdots,f_{n-3},U_1,\cdots,U_r,h'_1,\cdots,h'_r\}.$$
 So it is generated by $n+2r$ elements.
\end{enumerate}
Therefore, in each case, using \cite[Theorem 10.2]{E}, we conclude that $R''$ is a finite $W(\FF)$-algebra of Krull dimension $1$.

Let $P$ be a minimal prime of $R''$. So $R''/P$ is an integral domain which is finite over $W(\FF)$. Hence, it is isomorphic to a subring of $\overline{\QQ_p}$. So fix an inclusion $i : R''/P \to \overline{\QQ_p}$.
Let $\rho : G_{\QQ,Mp} \to \GL_2(R''/P)$ be the representation obtained by composing $\rho_M^{\univ}$ with the natural surjective map $R(M)_{\rhob_c} \to R''/P$. 
Using the inclusion $i$, we can view $\rho$ as a representation over $\overline{\QQ_p}$.
Note that there exists a minimal prime of $Q$ of $R:=R(M)_{\rhob_c}/\mathcal{I}_0$ such that the quotient map $R \to R''$ induces an isomorphism $R/Q \simeq R''/P$.
Therefore, by Proposition~\ref{modularprop}, we get that $\rho$ is the $p$-adic Galois representation attached to a newform of level $M'$ such that $p \nmid M'$ and $\ell_i \mid\mid M'$ for all $1 \leq i \leq r$.

Following the proof of part~\eqref{bhaag1} of the theorem, we get that $\ell^2 \mid M'$. 
Note that $p \nmid \phi(N\ell^2)$ and $N\ell^2$ satisfies Hypotheses~\eqref{hyp2} and \eqref{hyp3} of Set-up~\ref{setup} .
Hence, it follows, from the proof of Theorem of Theorem~\ref{thmb}, that $M' \mid N\ell^2\prod_{i=1}^{r}\ell_i$.
This finishes the proof of part~\ref{bhaag2} of the theorem.

\section{Proofs of Corollaries}
\label{proofcor}

\begin{proof}[Proof of Corollary~\ref{corf}]
If $\text{Cl}(\QQ(\zeta_p))/p\text{Cl}(\QQ(\zeta_p))[\omega_p^{k_0}]=0$, then \cite[Lemma 21]{BK} implies that $\dim(H^1(G_{\QQ,p},\omega_p^{1-k_0}))=1$.
If $p \mid B_{k_0}$, then Herbrand--Ribet theorem implies that $\dim(H^1_{\{p\}}(G_{\QQ,p},\omega_p^{1-k_0})) \geq 1$.
Therefore, we conclude that $\dim(H^1_{\{p\}}(G_{\QQ,p},\omega_p^{1-k_0})) = 1$.
Hence, Theorem~\ref{thmb} for $N=1$, along with the assumption that $p \nmid \prod_{i=1}^{r}(\ell_i-1)$, implies Corollary~\ref{corf}.
\end{proof}

\begin{proof}[Proof of Corollary~\ref{corb}]
By Herbrand-Ribet theorem, $p \nmid B_{k_0}$ if and only if $$\text{Cl}(\QQ(\zeta_p))/p\text{Cl}(\QQ(\zeta_p))[\omega_p^{1-k_0}]=0.$$
Hence, if $p \nmid B_{k_0}$, then $H^1_{\{p\}}(G_{\QQ,p},\omega_p^{1-k_0})=0$.
Therefore, Theorem~\ref{thmc} for $N=1$, along with the assumption that $p \nmid \prod_{i=0}^{r}(\ell_i-1)$, implies Corollary~\ref{corb}.
\end{proof}

\begin{proof}[Proof of Corollary~\ref{squarecor}]
Since $p \nmid B_{k_0}$, it follows, from the proof of Corollary~\ref{corb}, that $H^1_{\{p\}}(G_{\QQ,p},\omega_p^{1-k_0})=0$.
Therefore, Theorem~\ref{thme} for $N=1$ implies Corollary~\ref{squarecor}.
\end{proof}


\end{document}